\documentclass[10pt]{amsart}

\usepackage{amsmath,amssymb,amsthm,amsfonts,mathrsfs}
\usepackage{amscd,stmaryrd}
\usepackage[usenames,dvipsnames]{color}
\usepackage{hyperref}
\usepackage{graphicx,subfigure}
\definecolor{nicegreen}{RGB}{0,180,0}
\definecolor{purple}{RGB}{160, 0, 20}
\hypersetup{colorlinks=true,citecolor=nicegreen,linkcolor=Red,urlcolor=blue,pdfstartview=FitH}

\usepackage[divide={2.45cm,*,2.45cm}]{geometry}
\usepackage{enumerate}
\usepackage{color}
\usepackage{url}
\usepackage[all]{xy}

\allowdisplaybreaks

\newtheorem{thm}{Theorem}[section]
\newtheorem*{thmnonum}{Theorem}
\newtheorem{cor}[thm]{Corollary}
\newtheorem{lemma}[thm]{Lemma}
\newtheorem{sublemma}[thm]{Sublemma}
\newtheorem{propn}[thm]{Proposition}

\theoremstyle{definition}
\newtheorem{defn}[thm]{Definition}

\newtheorem{remark}[thm]{Remark}

\theoremstyle{plain}
\newtheorem*{notation}{Notation}

\newcommand{\fpb}{\overline{\mathbb{F}}_p}

\newcommand{\kk}{\mathfrak{k}}
\newcommand*{\longhookrightarrow}{\ensuremath{\lhook\joinrel\relbar\joinrel\rightarrow}}
\newcommand*{\longtwoheadrightarrow}{\ensuremath{\relbar\joinrel\twoheadrightarrow}}
\newcommand{\sub}[2]{\genfrac{}{}{0pt}{}{#1}{#2}}

\newcommand{\ind}{\textnormal{ind}}
\newcommand{\Ind}{\textnormal{Ind}}
\newcommand{\End}{\textnormal{End}}
\newcommand{\Hom}{\textnormal{Hom}}
\newcommand{\Ord}{\textnormal{Ord}}
\newcommand{\s}{\textnormal{S}}
\newcommand{\St}{\textnormal{St}}
\newcommand{\chps}{\chi_\psi}
\newcommand{\pr}{\mathsf{pr}}

\newcommand{\cA}{{\mathcal{A}}}

\newcommand{\cC}{{\mathcal{C}}}

\newcommand{\cE}{{\mathcal{E}}}

\newcommand{\cH}{{\mathcal{H}}}

\newcommand{\cO}{{\mathcal{O}}}

\newcommand{\cS}{{\mathcal{S}}}

\newcommand{\bK}{{\mathbf{K}}}

\newcommand{\bbC}{{\mathbb{C}}}

\newcommand{\bbQ}{{\mathbb{Q}}}

\newcommand{\bbZ}{{\mathbb{Z}}}

\newcommand{\tB}{\widetilde{B}}

\newcommand{\tF}{\widetilde{F}}
\newcommand{\tG}{\widetilde{G}}
\newcommand{\tH}{\widetilde{H}}

\newcommand{\tK}{\widetilde{K}}
\newcommand{\tL}{\widetilde{L}}
\newcommand{\tM}{\widetilde{M}}

\newcommand{\tP}{\widetilde{P}}
\newcommand{\tQ}{\widetilde{Q}}

\newcommand{\tS}{\widetilde{S}}
\newcommand{\tT}{\widetilde{T}}
\newcommand{\tU}{\widetilde{U}}
\newcommand{\tV}{\widetilde{V}}

\begin{document}

\title{Irreducible admissible mod-$p$ representations of metaplectic groups}
\date{\today}
\author{Karol Kozio\l\ and Laura Peskin}
\address{Department of Mathematics, University of Toronto, 40 St. George Street, Toronto, ON M5S 2E4, Canada} \email{karol@math.toronto.edu}
\address{Department of Mathematics, Weizmann Institute of Science, 234 Herzl Street, Rehovot, 7610001, Israel}
\email{laura.peskin@weizmann.ac.il}

\begin{abstract}
Let $p$ be an odd prime number, and $F$ a nonarchimedean local field of residual characteristic $p$.  We classify the smooth, irreducible, admissible genuine mod-$p$ representations of the twofold metaplectic cover $\widetilde{\textnormal{Sp}}_{2n}(F)$ of $\textnormal{Sp}_{2n}(F)$ in terms of genuine supercuspidal (equivalently, supersingular) representations of Levi subgroups. To do so, we use results of Henniart--Vign\'{e}ras as well as new technical results to adapt Herzig's method to the metaplectic setting. As consequences, we obtain an irreducibility criterion for principal series representations generalizing the complete irreducibility of principal series representations in the rank 1 case, as well as the fact that irreducibility is preserved by parabolic induction from the cover of the Siegel Levi subgroup.  
\end{abstract}

\maketitle

\section{Introduction}

\subsection{Context}
Let $p$ be an odd prime number, and $F$ a nonarchimedean local field of residual characteristic $p$.  The irreducible admissible mod-$p$ representations of a connected reductive group over $F$ have recently been classified up to supercuspidals by Abe--Henniart--Herzig--Vign\'eras \cite{abehenniartherzigvigneras:irredmodp}, completing a line of research begun by Barthel--Livn\'e in the 1990s (\cite{barthellivne:irredmodp}, \cite{barthellivne:ordunram}).  Their classification is part of an ongoing effort to formulate a mod-$p$ local Langlands correspondence, which aims to generalize the existing 
correspondence, due to Breuil \cite{breuil:modpgl2i}, between certain two-dimensional mod-$p$ representations of $\text{Gal}(\overline{\bbQ}_p/\bbQ_p)$ and certain admissible mod-$p$ representations of $\textnormal{GL}_2(\bbQ_p)$. It is expected that such a correspondence would be compatible with a $p$-adic local Langlands correspondence; this is the case for $\textnormal{GL}_2(\bbQ_p)$, where the $p$-adic correspondence is due to Colmez, Kisin, Pa\v{s}k\={u}nas, and others (see \cite{colmez:gl2qpphigamma}, \cite{kisin:defs}, \cite{paskunas:colmezfunctor}; see also \cite{cdp:padicllc} for a recent treatment, and \cite{berger:padicllc} for an overview). Furthermore, it is expected that a mod-$p$ local Langlands correspondence would reflect information about mod-$p$ automorphic forms via generalizations of Serre's conjectures.

In the setting of $\bbC$-representations, recent advances have raised the possibility of incorporating covers of reductive algebraic groups into the local Langlands correspondences. So far, the results are most complete in the case of the metaplectic double cover of $\textnormal{Sp}_{2n}(F)$.  In this case, Gan--Savin \cite{gansavin:epsilon} produced a correspondence between $\bbC$-representations of the metaplectic double cover of $\textnormal{Sp}_{2n}(F)$ and certain Weil-Deligne representations by composing the theta correspondence (properties of which were established in the necessary generality in \emph{loc. cit.}) with the local Langlands correspondence for odd special orthogonal groups. More recently, work of Weissman (\cite{weissman:metaplecticlgps}, \cite{weissman:lgpsparams}), Gan--Gao \cite{gangao:langlandsweissman}, and others has laid groundwork for conjectural local Langlands correspondences of $\bbC$-representations for a general class of covering groups. In particular, Weissman has constructed a candidate for the $L$-group of a Brylinski--Deligne cover of a quasisplit algebraic group, whose specialization to Gan--Savin's setting explains some surprising features of the local Langlands correspondence for the metaplectic cover of $\textnormal{Sp}_{2n}(F)$ (see \cite{gangao:langlandsweissman} $\S$12). 

As the two aforementioned programs develop, it is natural to ask what role the modular representations of covering groups might play. In the mod-$\ell$ setting, for $\ell$ sufficiently large, M\'inguez \cite{minguez:typeiihowe} has shown the existence of a bijective theta correspondence for type II dual reductive pairs over a nonarchimedean local field of residual characteristic $p \neq \ell$; he also found that the analogous map can fail to be bijective for certain small values of $\ell$. In the mod-$p$ setting, Shin \cite{shin:avweilrepn} notes that several key ingredients of the classical theta correspondence are no longer available at all, and he proposes new geometric constructions to take their place in a hypothetical mod-$p$ theta correspondence. In particular, he constructs algebraic Weil representations of metaplectic group schemes which agree with the classical objects upon taking $\bbC$-coefficients, but whose specializations to $\fpb$-coefficients depend on the input of a $p$-divisible group and are \textit{not}, in general, representations of the metaplectic double cover of $\textnormal{Sp}_{2n}(F)$ considered here.  Computations using Shin's constructions \cite{shin:modptheta} raise intriguing questions, and probably require more data (particularly for type I dual reductive pairs) to interpret. 

We aim to provide some initial information regarding the mod-$p$ representation theory of metaplectic groups over local fields, in the hope that this will help shape hypothetical mod-$p$ versions of the above correspondences.

\subsection{Main results}
Our main result is a classification of the smooth, irreducible, admissible, genuine mod-$p$ representations of the metaplectic double cover of $\textnormal{Sp}_{2n}(F)$ in terms of supercuspidal representations (from now on, ``representation'' will always mean ``smooth $\fpb$-representation'').  This generalizes a classification by the second-named author \cite{peskin:mp2modp} for the metaplectic double cover of $\textnormal{SL}_2(F)$. We will denote $\textnormal{Sp}_{2n}(F)$ by $G$ and its metaplectic cover by $\tG$, and we refer the reader to $\S$\ref{cover} for a construction of $\tG$. Here we note only that there is an open, continuous surjection $\pr: \tG \longrightarrow G$ with kernel isomorphic to $\{\pm 1\}$, and that a \textit{genuine} representation of $\tG$ is one which does not factor through $\pr$.  We also fix a maximal torus $T$ in $G$, and a Borel subgroup $B$ containing $T$.  Given any closed subgroup $H$ of $G$, we let $\tH$ denote the subgroup $\pr^{-1}(H)$ of $\tG$.

We closely follow the existing classifications of irreducible admissible representations of $p$-adic reductive groups (\cite{abe:irredmodp} in the split case, \cite{abehenniartherzigvigneras:irredmodp} in general) which build on the techniques introduced by Herzig in \cite{herzig:modpgln}. In particular, our classification is expressed in terms of \textit{supersingular triples} $(\tP^-, \sigma, \tQ^{-})$. Here $P$ is a standard parabolic subgroup of $G$ (that is, a parabolic subgroup containing $B$), $\sigma$ is a genuine \textit{supersingular} representation (we will discuss this term presently) of the Levi factor of $\tP$, and $\tQ^-$ is the preimage in $\tG$ of a standard parabolic subgroup of $G$ satisfying some conditions with respect to $\tP^-$ and $\sigma$.

Given a supersingular triple $(\tP^-, \sigma, \tQ^-)$, we define a certain parabolically induced genuine $\tG$-representation, denoted by $I(\tP^-, \sigma, \tQ^-)$. After defining two supersingular triples $(\tP^-, \sigma, \tQ^-)$, $(\tP'^-, \sigma', \tQ'^-)$ to be equivalent if $\tP^- = \tP'^-$, $\tQ^- = \tQ'^-$, and $\sigma \cong \sigma'$, we obtain the following classification as our main result:

\begin{thmnonum}[Theorem \ref{classn}] The map $(\tP^-,\sigma,\tQ^-)\longmapsto I(\tP^-,\sigma,\tQ^-)$ gives a bijection between equivalence classes of supersingular triples and isomorphism classes of irreducible admissible genuine representations of $\tG$. 
\end{thmnonum}

We now discuss the ingredients needed to state the classification more precisely. \\

\noindent 1.~\textit{Supersingular representations.} We refer the reader to Definition \ref{supersingular} for the precise definition of supersingularity in our context. The notion of a \textit{supersingular} representation was introduced by Barthel--Livn\'e (\cite{barthellivne:ordunram}, \cite{barthellivne:irredmodp}) in the case of $\textnormal{GL}_2(F)$, and a generalized definition for split reductive groups was made by Herzig (\cite{herzig:modpgln} $\S$1.2.1) using the mod-$p$ Satake transform of \cite{herzig:modpsatake}. The central piece of the definition is a condition on the systems of eigenvalues which occur in weight spaces of a representation under the action of the associated spherical Hecke algebras: no such system of eigenvalues may factor through a Satake transform to a spherical Hecke algebra of any proper Levi subgroup. Herzig also requires supersingular representations to be irreducible and admissible (admissibility was not required by \cite{barthellivne:ordunram}, \cite{barthellivne:irredmodp}), and shows that supersingularity is equivalent to supercuspidality for this class of representations.  We essentially adopt Herzig's definition (adding a genuineness condition), using results of Henniart--Vign\'eras \cite{henniartvigneras:modpsatake} which guarantee that Herzig's mod-$p$ Satake transform adapts well to the metaplectic case. For the appropriate definition of supercuspidality in our case, supersingularity is again equivalent to supercuspidality (see below).\\


\noindent 2.~\textit{Generalized Steinberg representations.} Let $P \subset Q$ be two standard parabolic subgroups of $G$, and let $P^{-} \subset Q^{-}$ denote the respective opposite parabolic subgroups. The \textit{generalized Steinberg representation} associated to the pair $(P^{-}, Q^{-})$ is defined to be 
$$\St_{P^-}^{Q^-}:= \Ind_{P^-}^{Q^-}(1)\Big/\sum_{Q'}\Ind_{Q'^-}^{Q^-}(1),$$
where the sum runs over standard parabolic subgroups $Q'$ such that $P^{-} \subsetneq Q'^{-} \subset Q^{-}$, and where $1$ denotes the trivial representation of each group $Q'$ in turn. Each $\St_{P^{-}}^{Q^-}$ is an irreducible admissible representation of $Q^-$ (see \cite{ly:steinberg}, Th\'eor\`eme 3.1, and note that the unipotent radical of $Q^-$ acts trivially). In the following, we will let $\St_{P^{-}}^{Q^{-}}$ denote the representation of $\tQ^-$ obtained by inflating the generalized Steinberg representation of $Q^-$ (note that this inflation is not a genuine representation of $\tQ^-$). \\

\noindent 3.~\textit{Supersingular triples.} Let $P$ be a standard parabolic subgroup of $G$ with Levi factor $M$, and let $\sigma$ be a genuine supersingular representation of $\tM$. We define a certain standard Levi subgroup $M(\sigma) \supset M$ of $G$ such that $\sigma$ extends uniquely to a genuine representation ${}^e\sigma$ of the preimage $\tM(\sigma)$ (see $\S$\ref{classnchapter} for an explicit definition of $M(\sigma)$). Let $P(\sigma)$ be the standard parabolic subgroup of $G$ whose Levi factor is $M(\sigma)$, and let $Q$ be another standard parabolic subgroup of $G$. We say that $(\tP^{-}, \sigma, \tQ^{-})$ is a \textit{supersingular triple} if $\tP^{-} \subset \tQ^{-} \subset \tP(\sigma)^{-}$. \\

\noindent 4.~\textit{$\tG$-representation associated to a supersingular triple.} Let $(\tP^-,\sigma,\tQ^-)$ be a supersingular triple, let $\tP(\sigma)$ denote the parabolic subgroup of $\tG$ mentioned in the previous point, and inflate ${}^e\sigma$ from $\tM(\sigma)$ to $\tP(\sigma)^{-}$. We then define 
$$I(\tP^-,\sigma,\tQ^-) := \Ind_{\tP(\sigma)^-}^{\tG}\left({}^e\sigma\otimes\St_{Q^-}^{P(\sigma)^-}\right).$$
Following \cite{abehenniartherzigvigneras:irredmodp}, we show in Proposition \ref{irred} that $I(\tP^{-}, \sigma, \tQ^{-})$ is irreducible, admissible, and genuine whenever $(\tP^-,\sigma,\tQ^-)$ is a supersingular triple.\\

Using the classification theorem above, we obtain several consequences.  We say that an irreducible admissible genuine representation $\pi$ of $\tG$ is \textit{supercuspidal} if $\pi$ is not isomorphic to a subquotient of the parabolic induction of an irreducible admissible genuine representation of a proper Levi subgroup of $\tG$.  We then have:

\begin{thmnonum}[Proposition \ref{ssscG}]
Let $\pi$ be an irreducible admissible genuine representation of $\tG$.  Then $\pi$ is supersingular if and only if $\pi$ is supercuspidal.  
\end{thmnonum}

Next, let $P_\s$ denote the (standard) Siegel parabolic subgroup, and $M_{\s}\cong\textnormal{GL}_n(F)$ its standard Levi factor.  The classification theorem above, together with the analogous result in the reductive case, implies that parabolic induction from $\tM_{\s}$ preserves irreducibility:

\begin{thmnonum}[Lemma \ref{msirred}]
Let $\sigma$ be an irreducible admissible genuine representation of $\tM_{\s}$.  The parabolic induction $\Ind_{\tP_{\s}^{-}}^{\tG}(\sigma)$ is then an irreducible admissible genuine representation of $\tG$. 
\end{thmnonum}

In addition to the above, we also obtain the following information about the genuine principal series representations of $\tG$.  Here $\Pi$ denotes the set of simple roots of $G$ with respect to $T$ and $B$, and for $\alpha\in \Pi$, $\widetilde{\alpha}^{\vee}$ denotes a certain canonical map $F^{\times} \longrightarrow \tT$ (which is \textit{not} a homomorphism in general) lifting the coroot $\alpha^{\vee}: F^\times \longrightarrow T$; see $\S$\ref{splittings}.

\begin{thmnonum}[Corollary \ref{irredps}]
Let $\sigma$ be a genuine character of $\tT$. 
\begin{enumerate}
\item The length of $\Ind_{\tB^{-}}^{\tG}(\sigma)$ is at most $2^{n-1}$, and is equal to $2^{n-1}$ if and only if $\sigma \circ \widetilde{\alpha}^{\vee}(x) = 1$ for every $x\in F^\times$ and every short root $\alpha\in \Pi$.
\item The representation $\Ind_{\tB^{-}}^{\tG}(\sigma)$ is irreducible if and only if, for every short root $\alpha\in \Pi$, there exists $x_\alpha\in F^\times$ such that $\sigma \circ \widetilde{\alpha}^{\vee}(x_\alpha) \neq 1$. 
\item Let $\sigma'$ be a genuine character of $\tT$ such that $\sigma \neq \sigma'$. Then $\Ind_{\tB^{-}}^{\tG}(\sigma)$ and $\Ind_{\tB^{-}}^{\tG}(\sigma')$ are inequivalent. 
\end{enumerate}
\end{thmnonum}

We thus get an irreducibility criterion for principal series representations of $\tG$ which involves fewer conditions than the analogous criterion for $G$ (compare \cite{abe:irredmodp}, Theorem 1.3): in our criterion, no role is played by the restriction of $\sigma$ to the image of $\widetilde{\alpha}^{\vee}$ for the long root $\alpha \in \Pi$. Our irreducibility criterion generalizes the fact (\cite{peskin:mp2modp}) that $\widetilde{\textnormal{SL}}_2(F)$ has no degenerate mod-$p$ principal series representations, and therefore no obvious mod-$p$ analog of the nonsupercuspidal elementary Weil representation.

\subsection{Techniques}

In most respects, the methods of Herzig \cite{herzig:modpgln}, Abe \cite{abe:irredmodp}, and Abe--Henniart--Herzig--Vign\'eras \cite{abehenniartherzigvigneras:irredmodp} for reductive groups over $F$ turn out to go through for the metaplectic group $\tG$ with only minor adaptations. The work of Henniart--Vign\'eras in \cite{henniartvigneras:modpsatake} and \cite{henniartvigneras:cptparabolic} establishes many of the technical ingredients of Herzig's method for a large class of groups including $\tG$; others are easily adapted from the reductive case (see $\S$\ref{subsec:satake}).  Consequently, whenever a proof requires only minor cosmetic changes from the analogous proof in the reductive case, we often simply give a reference to the literature.  For the sake of completeness, however, we have attempted to give a more detailed account of the classification proof ($\S$\ref{classnchapter}).

The main difference between the metaplectic and reductive cases appears in the \textit{change-of-weight} step, in which one finds a criterion for the existence of an isomorphism 
$$\chi\otimes_{\cH_{\tG}(\tV)}\ind_{\tK}^{\tG}(\tV) \stackrel{\sim}{\longrightarrow} \chi\otimes_{\cH_{\tG}(\tV')}\ind_{\tK}^{\tT}(\tV').$$
Here, $\tV$ and $\tV'$ are irreducible genuine representations of the maximal compact subgroup $\tK \subset \tG$, $\cH_{\tG}(\tV)$ denotes the Hecke algebra of $\tG$-intertwiners of the compact induction $\ind_{\tK}^{\tG}(\tV)$ (and likewise for $\tV'$), and $\chi$ denotes an $\fpb$-character of $\cH_{\tG}(\tV)$ (viewed also as a character of $\cH_{\tG}(\tV')$ via a natural identification of the two algebras).  We find a criterion (Theorem \ref{thm:changeofwt}) for the existence of such an isomorphism which is strictly weaker than the analogous criterion for $G$. (See \cite{abe:irredmodp}, Theorem 4.1 for the change-of-weight criterion applicable to $G$.) This weakening of the change-of-weight criterion is responsible for the relative weakness of the irreducibility criteria of Lemma \ref{msirred} and Corollary \ref{irredps}, as compared to the criteria for representations of $G$. 

The proof of our change-of-weight criterion is given in $\S$\ref{ch:changeofwt}, and involves computing the Satake transform of a certain Hecke operator. In $\S$\ref{classnchapter}, after adapting the notion of supersingularity to our context and modifying the change-of-weight criterion as per the result of $\S$\ref{ch:changeofwt}, we continue with the classification as in the reductive case. \\

\noindent\textbf{Acknowledgements.}  The authors would like to thank Florian Herzig for suggesting the alternate proof of Proposition \ref{propn:st2l} in the case of a short simple root.  Some of this work was carried out while the first-named author visited the University of British Columbia, and both authors wish to thank Rachel Ollivier for her support via an NSERC Discovery Grant. The first-named author was supported by NSF grant DMS-1400779 and an EPDI fellowship. The authors thank the anonymous referee for several useful comments.

\section{Notation and Preliminary Results}

Let $p$ be an odd prime and $F$ a nonarchimedean local field of residual characteristic $p$.  We let $\cO$ denote the ring of integers of $F$, $\varpi\in \cO$ a fixed uniformizer, and $\kk$ the residue field $\cO/\varpi\cO$.  We fix an algebraic closure $\overline{\kk}$ of $\kk$, and denote the order of $\kk$ by $q$.  Furthermore, we let $\mu_2 := \mu_2(F) = \{\pm 1\}$ denote the group of square roots of unity in $F$, and let
$$(-,-)_F:F^\times \times F^\times\longrightarrow \mu_2$$
denote the quadratic Hilbert symbol.

\subsection{Symplectic groups}
Fix an integer $n \geq 1$ and let $G := \mathbf{Sp}_{2n}$ denote the symplectic group of rank $n$, defined and split over $F$.  We will abuse notation and write $G$ for $G(F)$, the group of $F$-points (and similarly for other algebraic groups over $F$).  We fix a split maximal torus $T$, and we choose a hyperspecial point in the corresponding apartment of the Bruhat--Tits building of $G$.  Such a point gives a connected reductive integral model of $G$ over $\cO$, and we continue to denote this model by $G$.  We set $K := G(\cO)$, which is a hyperspecial maximal compact subgroup.  The finite group of $\kk$-points of $G$ will be denoted $G(\kk)$.

Let $X^*(T)$ (respectively, $X_*(T)$) denote the group of algebraic characters (resp., cocharacters) of $T$ (or, more precisely, of the algebraic group defining $T$). Let $\Phi \subset X^*(T)$ denote the root system of $G$ with respect to $T$; thus $\Phi$ is a root system of type $C_n$.  We fix a set of simple roots 
$$\Pi := \{\alpha_i : 1 \leq i \leq n\},$$
labeled so that $\alpha_n$ is the unique long simple root and so that the roots $\alpha_i$ and $\alpha_{i + 1}$ are adjacent in the Dynkin diagram of $G$.  We fix a $\bbZ$-basis $\{ \chi_i : 1 \leq i \leq n\}$ for $X^*(T)$, where $2\chi_n := \alpha_n$ and $\chi_i := \alpha_i + \chi_{i+1}$ for $1 \leq i \leq n-1$. Let 
$$\langle - , -\rangle:X^*(T) \times X_*(T) \longrightarrow \bbZ$$ 
denote the natural perfect pairing, and let $\alpha_i^{\vee} \in X_*(T)$ denote the coroot corresponding to $\alpha_i$ for $1 \leq i \leq n$ (normalized so that $\langle \alpha_i, \alpha_i^{\vee}\rangle = 2$).  We fix a $\bbZ$-basis $\{ \lambda_i : 1 \leq i \leq n\}$ for $X_*(T)$, where $\lambda_n := \alpha_n^{\vee}$ and $\lambda_i := \alpha_i + \lambda_{i + 1}$ for $1 \leq i \leq n-1$.  Note that $\langle\chi_i, \lambda_j\rangle = \delta_{i,j}$.

The simple roots define a partition $\Phi = \Phi^+ \sqcup\Phi^-$ and a partial order $\leq$ on $X_*(T)$: we say $\lambda \leq \mu$ if $\mu - \lambda = \sum_{i = 1}^n a_i\alpha_i^\vee$, where $a_i\in \bbZ_{\geq 0}$.  Moreover, the set $\Pi$ also defines the monoid $X_*(T)_-$ of antidominant cocharacters, given by
$$X_*(T)_- := \{\lambda\in X_*(T): \langle\alpha,\lambda\rangle \leq 0 \textnormal{ for all }\alpha\in \Pi\}.$$

For $\alpha\in \Phi$, we let $U_\alpha$ denote the root subgroup associated to $\alpha$, and let $u_\alpha:F\stackrel{\sim}{\longrightarrow}U_\alpha$ denote a fixed root morphism such that the set $\{u_{\alpha}: \, \alpha \in \Phi\}$ satisfies the properties of Lemma 8.1.4 of \cite{springer:linalggrps}. In particular, $tu_\alpha(x)t^{-1} = u_\alpha(\alpha(t)x)$ for $t\in T$, $x\in F$.  Furthermore, since $G$ is (the group of $F$-points of) a Chevalley group, we may and will choose our root morphisms so that all structure constants $c_{\alpha, \beta; j, k}$ appearing in the commutator formula (for $\beta\neq \pm \alpha$)
\begin{equation}
\label{ucomm}
[u_{\alpha}(x), u_{\beta}(y)] := u_{\alpha}(x)u_{\beta}(y)u_{\alpha}(x)^{-1}u_{\beta}(y)^{-1} = \prod_{\substack{j, k > 0\\ j\alpha + k\beta \in \Phi}} u_{j\alpha + k\beta}(c_{\alpha, \beta; j, k}x^jy^k)
\end{equation}
belong to $\bbZ$ (where the product is taken with respect to some fixed total order on $\Phi$).  We will also assume that the subgroup $K$ is generated by $\{ u_{\alpha}(x): \, \alpha \in \Phi, \, x \in \cO\}$.


Let $B$ denote the Borel subgroup of $G$ with respect to $\Pi$, and let $U$ denote the unipotent radical of $B$.  Thus $B$ is generated by $T$ and $U_\alpha$ for $\alpha\in \Phi^+$.  We let $B^- = T\ltimes U^-$ denote the opposite Borel subgroup.  More generally, the notation $P = M\ltimes N$ will be used to denote parabolic subgroups of $G$, while $P^- = M\ltimes N^-$ denotes the opposite parabolic.  We say $P = M\ltimes N$ is a \textit{standard parabolic subgroup} if $P$ is a parabolic subgroup containing $B$, $N$ is its unipotent radical contained in $U$, and $M$ is a Levi subgroup containing $T$. Moreover, we say a subgroup of $G$ is a \textit{standard Levi subgroup} if it is the Levi factor of a standard parabolic subgroup.

Standard parabolic subgroups correspond bijectively to subsets of $\Pi$.  For a subset $J\subset \Pi$, we let $P_J = M_J\ltimes N_J$ denote the corresponding standard parabolic subgroup.  Reciprocally, for a standard parabolic subgroup $P = M\ltimes N$, we will denote by $\Pi_M$ the corresponding subset of $\Pi$. For $\lambda \in X_*(T)_-$, we will write $P_{\lambda} = M_{\lambda} \ltimes N_{\lambda}$ for standard parabolic subgroup corresponding to the subset $\{\alpha\in \Pi: \langle \alpha, \lambda \rangle = 0\};$ see \cite{springer:linalggrps}, $\S$13.4.2 and 15.4.4.  We define the \textit{Siegel parabolic subgroup} $P_\s = M_\s\ltimes N_\s$ as the standard parabolic subgroup corresponding to the subset 
$$\Pi_{\s} := \{\alpha_1,\ldots, \alpha_{n - 1}\}.$$  
The Siegel Levi subgroup $M_{\s}$ is isomorphic to $\textnormal{GL}_n(F)$.

We will also need to consider the symplectic similitude group, so we set notation for it here. Let $\textnormal{GSp}_{2n}(F)$ be the group of $F$-points of $\mathbf{GSp}_{2n}$, defined and split over $F$, and chosen such that $\mathbf{Sp}_{2n}\subset\mathbf{GSp}_{2n}$.  The group $\textnormal{GSp}_{2n}(F)$ is the group of linear transformations of a $2n$-dimensional vector space over $F$ which preserve a fixed symplectic form up to a nonzero scalar.  The torus $T$ of $G$ is contained in a (split) maximal torus $T_{G}$ of $\text{GSp}_{2n}(F)$, and we let $\lambda_{n+1}$ denote a fixed cocharacter of $T_G$ such that $\{\lambda_1, \dots, \lambda_n, \lambda_{n+1}\}$ is a $\bbZ$-basis for $X_*(T_{G})$. We then have $\text{GSp}_{2n}(F) \cong \text{Sp}_{2n}(F) \rtimes \lambda_{n+1}(F^\times)$.



\subsection{Covering groups}\label{cover}

We now discuss covers of $\textnormal{Sp}_{2n}(F)$ and $\textnormal{GSp}_{2n}(F)$.  Our main reference is \cite{gangao:langlandsweissman} (especially $\S$16.1 and $\S$16.3), which in turn is based on \cite{brylinskideligne:extnsbyk2}.

The general framework of \cite{brylinskideligne:extnsbyk2} shows how to construct central extensions of the form
$$1 \longrightarrow \bK_2\longrightarrow \overline{\mathbf{GSp}}_{2n}\longrightarrow\mathbf{GSp}_{2n}\longrightarrow 1$$
as sheaves of groups on the big Zariski site of $\textnormal{Spec}(F)$, where $\bK_2$ is the sheaf of groups associated to $K_2$ in Quillen's $K$-theory.  By Theorem 6.2 of \textit{loc. cit.} (see also $\S$2.5 of \cite{gangao:langlandsweissman}), the category of such extensions is equivalent to a category whose objects are triples $(Q,\cE,f)$, where
$$Q:\bigoplus_{i = 1}^{n + 1}\bbZ\lambda_i\longrightarrow \bbZ$$
is a quadratic form invariant under the Weyl group of $\mathbf{GSp}_{2n}$, $\cE$ is a certain group extension, and $f$ is a certain isomorphism.  The precise definitions of $\cE$ and $f$ are not relevant for our purposes, but let us only mention that for every choice of $Q$ as above, there exist $\cE$ and $f$ for which the triple $(Q,\cE,f)$ satisfies the necessary coherence properties (cf. \cite{gangao:langlandsweissman}, $\S$2.6).

Let $Q$ denote the Weyl-invariant quadratic form defined by
$$Q\left(\sum_{i = 1}^{n + 1}a_i\lambda_i\right) = \sum_{i = 1}^n (a_i^2 + a_ia_{n + 1}).$$
Note in particular that $Q(\alpha_i^\vee) = 2$ if $1\leq i \leq n - 1$ and $Q(\alpha_n^\vee) = 1$.  We fix choices of $\cE$ and $f$ for which the triple $(Q, \cE, f)$ satisfies the relevant coherence conditions, and we let $\overline{\mathbf{GSp}}_{2n}$ denote the associated extension of $\mathbf{GSp}_{2n}$ by $\bK_2$.  Since $\textnormal{H}^1(F,\bK_2) = 0$, upon taking $F$-points we obtain
$$1 \longrightarrow K_2(F)\longrightarrow \overline{\mathbf{GSp}}_{2n}(F)\longrightarrow \textnormal{GSp}_{2n}(F)\longrightarrow 1,$$
and we define the topological group $\widetilde{\textnormal{GSp}}_{2n}(F)$ by pushing out the above exact sequence by the quadratic Hilbert symbol $(-,-)_F:K_2(F)\longrightarrow \mu_2$~:
$$1 \longrightarrow \mu_2\longrightarrow \widetilde{\textnormal{GSp}}_{2n}(F)\longrightarrow \textnormal{GSp}_{2n}(F)\longrightarrow 1.$$

We shall need a more explicit description of the maximal torus of $\widetilde{\textnormal{GSp}}_{2n}(F)$.  Fix two cocharacters $\lambda,\lambda'$ of the torus of $\textnormal{GSp}_{2n}(F)$, and let $x,y\in F^\times$.  Letting $\widetilde{\lambda}(x)$ and $\widetilde{\lambda}'(y)$ denote arbitrary lifts of $\lambda(x), \lambda'(y)$ to $\widetilde{\textnormal{GSp}}_{2n}(F)$, we obtain the following commutator formula (cf. \cite{gangao:langlandsweissman}, $\S$3.3):
\begin{equation}\label{comm}
\left[\widetilde{\lambda}(x),\widetilde{\lambda}'(y)\right] = (x,y)_F^{Q(\lambda + \lambda')-Q(\lambda) - Q(\lambda')}.
\end{equation}
In particular, if $\lambda = \sum_{i = 1}^n a_i\lambda_i$ and $\lambda' = \lambda_{n + 1}$, we obtain
\begin{equation}\label{gspconj}
\left[\widetilde{\lambda}(x),\widetilde{\lambda}'(y)\right] = (x,y)_F^{\sum_{i = 1}^n a_i},
\end{equation}
while if $\lambda = \sum_{i = 1}^n a_i \lambda_i$ and $\lambda' = \sum_{i = 1}^n a_i' \lambda_i$, then
\begin{equation}\label{spcomm}
\left[\widetilde{\lambda}(x), \widetilde{\lambda}'(y)\right] = (x, y)_F^{\sum_{i = 1}^n 2a_i a_i'} = 1.
\end{equation}

Finally, we define $\tG = \widetilde{\textnormal{Sp}}_{2n}(F)$ to be the pullback to $G$ of the cover $\widetilde{\textnormal{GSp}}_{2n}(F)\longrightarrow \textnormal{GSp}_{2n}(F)$:

\centerline{
\xymatrix{
1\ar[r] & \mu_2\ar[r]&  \widetilde{\textnormal{GSp}}_{2n}(F) \ar[r]\ar@{}[dr]|{\square} & \textnormal{GSp}_{2n}(F)\ar[r] & 1\\
1\ar[r] & \mu_2\ar[r]\ar@{=}[u] & \tG \ar^{\pr}[r]\ar[u] & G\ar[r]\ar[u] & 1
}
} 
\noindent Here, $\pr:\tG\longrightarrow G$ denotes the projection map, which is continuous and open.  We will identify $\mu_2$ with its image in $\tG$ via the injection above.

The group $\tG$ is isomorphic to the classical metaplectic double cover constructed by Weil (see $\S$16 of \cite{gangao:langlandsweissman}, $\S$4.3 of \cite{weissman:metaplecticlgps}, or Proposition 4.15 of \cite{brylinskideligne:extnsbyk2}).  When necessary, we will work with coordinates on $\tG$ as follows: as a set, we identify $\tG$ with $G\times \mu_2$, with product given by $(g, \zeta)\cdot (g', \zeta') = (gg', \sigma(g, g')\zeta\zeta')$.  Here, $\sigma$ denotes Rao's cocycle (cf. \cite{rao:weilrepn}), whose image in $\textnormal{H}^2(G,\mu_2)\cong \mu_2$ gives the nontrivial class.  In particular, if $m, m'\in M_\s$, we have
\begin{equation}\label{cocycleMs}
(m,\zeta)\cdot(m',\zeta') = \left(mm', (\sideset{}{_{M_\s}}\det(m), \sideset{}{_{M_\s}}\det(m'))_F\zeta\zeta'\right)
\end{equation}
(cf. \textit{loc. cit.}, Lemma 5.1(iii) and Corollary 5.5(2)).

We state one final fact, which is implicit in many calculations below.  Let $g,g'\in G$, and let $\widetilde{g}, \widetilde{g}'\in \tG$ be arbitrary choices of lifts.  By 2.II.1 Proposition and 2.II.5 Lemme of \cite{mvw:howebook}, the elements $g$ and $g'$ commute in $G$ if and only if $\widetilde{g}$ and $\widetilde{g}'$ commute in $\tG$.  Note that this is not true for $\textnormal{GSp}_{2n}(F)$ and its double cover (cf. formula \eqref{comm}).

\subsection{Splittings and distinguished elements}
\label{splittings}

The cover $\tG \longrightarrow G$ splits over certain subgroups of $G$, as follows.  If $N$ is any unipotent subgroup of $G$, then $\S$3.2 of \cite{gangao:langlandsweissman} shows that the extension splits uniquely over $N$.  Likewise, $\S$4.4 of \textit{loc. cit.} shows that the extension also splits uniquely over the maximal compact subgroup $K$ (our assumption that $p \neq 2$ is necessary here).  We remark that one may prove these results in a much more ``hands-on'' way by computing $\textnormal{H}^1(-,\mu_2)$ and $\textnormal{H}^2(-,\mu_2)$ (and utilizing $\S$11 of \cite{moore:gpextensions}).  Similar cohomological arguments show that the extension splits uniquely over $N\cap K$, and therefore the splittings over $N$ and $K$ must agree on $N\cap K$.

If $H$ is any closed subgroup of $G$, we will denote by 
$$\tH := \pr^{-1}(H)$$ 
the preimage of $H$ in $\tG$.  If $H$ is either $K$ or a unipotent subgroup of $G$ (or the intersection thereof), we will write $H^*$ for the image of $H$ under its unique splitting, which gives $\tH \cong H^*\times \mu_2$.  If $H$ is some closed subgroup of $K$ (for example, $M\cap K$ for some standard Levi $M$), we will use $H^*$ to denote the image of $H$ in $K^*$ by the splitting over $K$ (note that the splitting over such an $H$ is not unique in general).  Additionally, the factorization $P = M\ltimes N$ of a standard parabolic subgroup lifts to $\tP = \tM\ltimes N^*$, and we refer to $\tP$ as a standard parabolic subgroup of $\tG$.

Let $\alpha\in \Phi$.  Since the cover $\tG\longrightarrow G$ splits uniquely over $U_\alpha$, we may compose the root morphism $u_\alpha$ with the splitting to obtain a homomorphism
$$\widetilde{u}_\alpha:F\stackrel{\sim}{\longrightarrow} U_\alpha^*\subset \tU_\alpha.$$
Given this, for $\alpha \in \Phi$ and $x \in F^\times$ we define
$$\widetilde{\alpha}^\vee(x)  :=  \widetilde{u}_{\alpha}(x)\widetilde{u}_{-\alpha}(-x^{-1})\widetilde{u}_{\alpha}(x - 1)\widetilde{u}_{-\alpha}(1)\widetilde{u}_{\alpha}(-1).$$
The element $\widetilde{\alpha}^\vee(x)\in \tT$ is a preimage of $\alpha^\vee(x)\in T$.  Note that $\widetilde{\alpha}^\vee$ is \emph{not} a homomorphism in general.

Since the group $G$ is simply connected, we have $X_*(T) = \bigoplus_{i = 1}^n\bbZ\alpha_i^\vee$.  Given $\lambda\in X_*(T)$, we write $\lambda = \sum_{i = 1}^n a_i\alpha_i^\vee$, and define
$$\widetilde{\lambda}(\varpi) := \prod_{i = 1}^n\widetilde{\alpha}_i^\vee(\varpi)^{a_i}.$$
The element $\widetilde{\lambda}(\varpi)\in \tT$ is a preimage of $\lambda(\varpi)\in T$.  Moreover, since $\tT$ is commutative (cf. equation \eqref{spcomm}), the map $\lambda\longmapsto \widetilde{\lambda}(\varpi)$ is an injective homomorphism which induces an isomorphism $X_*(T) \stackrel{\sim}{\longrightarrow}\tT/(\tT\cap \tK)$.

\subsection{Representations and weights}
\label{repns}

All representations will have coefficients in $\fpb$ unless noted otherwise. Let $H$ be any closed subgroup of $G$ or of $\tG$, and $V$ a representation of $H$.  The representation $V$ is called \textit{smooth} if for every vector $v \in V$ the stabilizer of $v$ is an open subgroup of $H$.  All representations appearing will be assumed to be smooth.  We denote by $V^H$ (resp., $V_H$) the subspace of invariants (resp., coinvariants) under the action of $H$.  We say $V$ is \textit{admissible} if $V^J$ is finite dimensional over $\fpb$ for every open subgroup $J$ of $H$.  Finally, let $\varepsilon:\mu_2 \longrightarrow \fpb^\times$ denote the nontrivial character of $\mu_2$.  If $H$ is any closed subgroup of $G$ and $\tH$ its preimage in $\tG$, then a representation $V$ of $\tH$ is called \textit{genuine} if $\zeta \cdot v = \varepsilon(\zeta)v$ for all $\zeta \in \mu_2$.

A \textit{weight of $K$} is an irreducible representation of $K$, and a \textit{weight of $\tK$} is a genuine irreducible representation of $\tK$. Every weight of $\tK$ is of the form $\tV = V \boxtimes \varepsilon$, where $V$ is a weight of $K$ viewed as a representation of $K^*$ via the canonical splitting of the extension over $K$.

Let $X_q(T)$ denote the set of $q$-restricted weights of $T$, i.e., 
$$X_q(T) := \{\chi \in X^*(T): \, 0 \leq \langle \chi, \alpha^{\vee} \rangle < q \textnormal{ for all }  \alpha \in \Pi\},$$
and also define
$$X^0(T) := \{\chi \in X^*(T): \, \langle \chi, \alpha^{\vee}\rangle = 0 \textnormal{ for all } \alpha \in \Pi\}.$$
Given a character $\nu \in X^*(T)$ which satisfies $\langle\nu,\alpha^\vee\rangle \geq 0$ for all $\alpha\in \Pi$, we let $F(\nu)$ denote the associated algebraic representation of $G(\overline{\kk})$ of highest weight $\nu$.  By Proposition 2.2 of \cite{herzig:modpgln}, every weight of $\tK$ is equal to $\tF(\nu) := F(\nu) \boxtimes \varepsilon$ for some $\nu \in X_q(T)$, where we view $F(\nu)$ as a representation of $K$ via $K\longtwoheadrightarrow G(\kk)\longhookrightarrow G(\overline{\kk})$.  Moreover, $\tF(\nu) \cong \tF(\nu')$ if and only if $\nu - \nu' \in (q-1)X^0(T)$.  If $M$ is a standard Levi subgroup of $G$, then we will refer to genuine irreducible representations of $\tM \cap \tK$ as \textit{weights of $\tM \cap \tK$}, or just \textit{weights of $\tM$.} We denote the analogs of $X_q(T)$, $X^0(T)$, and $\tF(\nu)$ by $X^M_q(T)$, $X^0_M(T)$, and $\tF_{\tM}(\nu)$, respectively.

Finally, for $\nu\in X_q(T)$, we set 
$$\Pi_\nu := \{\alpha\in \Pi: \langle\nu,\alpha^\vee\rangle = 0\},$$
and given a standard Levi subgroup $\tM$, we say $\tV = \tF(\nu)$ is \textit{$\tM$-regular} if $\Pi_\nu\subset \Pi_M$.

Fix now two standard parabolic subgroups $P\subset Q$ of $G$, and let $M\subset L$ denote their Levi subgroups.  Let $N_1$ denote the intersection of $L$ with the unipotent radical of $P$, so that $\tM\ltimes N_1^{*}$ is a parabolic subgroup of $\tL$.  We then have the following lemma.

\begin{lemma}
\label{weights}
Let $\tV$ be a weight of $\tL$ and let $\tM \ltimes N_1^*$ be a standard parabolic of $\tL$. Then:
\begin{enumerate} 
\item $\tV^{(N_1 \cap K)^*}$ and $\tV_{(N_1^{-} \cap K)^*}$ are weights of $\tM$;
\item the natural, $\tM \cap \tK$-linear map $\tV^{(N_1 \cap K)^*} \longrightarrow \tV_{(N_1^{-} \cap K)^*}$ is an isomorphism;
\item $\tV^{(U \cap L \cap K)^*} \cong \tV_{(U^{-} \cap L \cap K)^*}$ is one-dimensional;
\item if $\tV = \tF_{\tL}(\nu)$ for some $\nu \in X_q^L(T)$, then $\tV^{(N_1 \cap K)^*} \cong \tF_{\tM}(\nu)$.
\end{enumerate}
\end{lemma}

\begin{proof}
This follows from \cite{herzig:modpgln}, Lemma 2.3 by identifying the underlying vector spaces of $\tV$ and $V$.  
\end{proof}

Now let $\tH\subset \tL$ denote two closed subgroups of $\tG$ which both contain the image of $\mu_2$ in $\tG$.  We let $\Ind_{\tH}^{\tL}(-)$ and $\ind_{\tH}^{\tL}(-)$ denote the functors of induction and compact induction, respectively, from the category of smooth $\tH$-representations to the category of smooth $\tL$-representations. These functors preserve genuineness. If $\tH$ is an open subgroup of $\tL$, then the functor $\ind_{\tH}^{\tL}(-)$ is exact, and Frobenius reciprocity gives an isomorphism $\Hom_{\tL}(\ind_{\tH}^{\tL}(\sigma), \pi) \cong \Hom_{\tH}(\sigma, \pi \big \vert_{\tH})$, natural in both arguments.

Assume now that $\tL$ is a standard Levi subgroup of $\tG$ and $\tH$ is a standard parabolic subgroup of $\tL$ with Levi subgroup $\tM$.  We view smooth representations of $\tM$ as smooth representations of $\tH^-$ via inflation. The functor $\Ind_{\tH^-}^{\tL}(-)$ from smooth $\tM$-representations to smooth $\tL$-representations is then an exact functor which preserves admissibility (cf. \cite{vigneras:rightadjoint}, Propositions 4.3 and 4.8).  We also have the functor of ordinary parts $\Ord_{\tH}^{\tL}(-)$, from the category of smooth $\tL$-representations to the category of smooth $\tM$-representations (cf. \textit{loc. cit.}).  This functor preserves admissibility and genuineness, and is a right adjoint to $\Ind_{\tH^-}^{\tL}(-)$ (\textit{loc. cit.}, Corollary 8.3).  Moreover, Corollary 8.4 of \textit{loc. cit.} shows that we have a functorial isomorphism $\sigma \cong \textnormal{Ord}_{\tH}^{\tL}\circ\Ind_{\tH^-}^{\tL}(\sigma)$ for any admissible genuine $\tM$-representation $\sigma$.

We adopt similar notations for subgroups of $G$.

\subsection{Subgroup decompositions in $\tG$}

We will need to make use of several subgroup decompositions in $\tG$.  In general, these are deduced from similar decompositions in the group $G$ using the map $\pr^{-1}$.  For instance, given a standard parabolic subgroup $P$ of $G$, we may lift the Iwasawa decomposition $G = P^-K$ via $\pr^{-1}$ to obtain
$$\tG = \tP^-\tK = \tP^-K^*.$$
Likewise the refined Cartan decomposition $G = \bigsqcup_{\lambda\in X_*(T)_-}K\lambda(\varpi)K$ lifts to
$$\tG = \bigsqcup_{\lambda\in X_*(T)_-}\tK\widetilde{\lambda}(\varpi)\tK,$$
and one may further refine this decomposition to
$$\tG = \bigsqcup_{\sub{\lambda\in X_*(T)_-}{\zeta\in\mu_2}}K^*\widetilde{\lambda}(\varpi)\cdot\zeta K^*$$
(note that the disjointness of the last union is implied by Theorem 9.2 of \cite{mcnamara:psreps}, whose proof remains valid in our situation).

\subsection{Recollections on the Satake map}
\label{subsec:satake}

We recall some basic properties of Satake maps.  Since the groups $\tG$, $\tP^-$, $\tK$, and the relevant Levi subgroups satisfy axioms (A1), (A2), 2.5(i), 2.5(ii), (C1), (C2), and the assumptions of $\S$2.8 of \cite{henniartvigneras:modpsatake}, we may apply the general formalism of $\S$2 of \textit{loc. cit.} and $\S$2 of \cite{henniartvigneras:cptparabolic}.

Fix two standard parabolic subgroups $P\subset Q$ of $G$, and let $M\subset L$ denote their respective Levi factors.  We again let $N_1$ denote the intersection of $L$ with the unipotent radical of $P$, so that $\tM\ltimes N_1^{*}$ is a parabolic subgroup of $\tL$.  We also fix a weight $\tV$ of $\tL\cap \tK$, and define
$$\cH_{\tL}(\tV):= \End_{\tL}\left(\ind_{\tL\cap\tK}^{\tL}(\tV)\right).$$
As in $\S$2 of \cite{henniartvigneras:cptparabolic}, we identify elements of $\cH_{\tL}(\tV)$ with functions $\varphi:\tL\longrightarrow \End_{\fpb}(\tV)$ which satisfy 

\begin{center}
\begin{itemize}
\item $\varphi(k_1\ell k_2) = k_1\circ \varphi(\ell)\circ k_2$, where $k_i\in \tL\cap \tK, \ell\in \tL$, and where the $k_i$ appearing on the right-hand side denote the corresponding automorphism of $\tV$;
\item the support of $\varphi$ is compact.  
\end{itemize}
\end{center}
The product structure on (this realization of) $\cH_{\tL}(\tV)$ is given by convolution.  By the refined Cartan decomposition, a function $\varphi\in\cH_{\tL}(\tV)$ is determined by its values on the elements $\widetilde{\lambda}(\varpi)$, where $\lambda\in X_*(T)_{L,-}:=\{\lambda\in X_*(T): \langle\alpha, \lambda\rangle\leq 0~\textnormal{for all}~\alpha\in \Pi_L\}$.  

By $\S$2 of \textit{loc. cit.}, we have a \textit{Satake morphism}
$$\cS_{\tL}^{\tM}:\cH_{\tL}(\tV)\longrightarrow \cH_{\tM}(\tV_{(N_1^-\cap K)^*}),$$
which, by Proposition 2.2 of \textit{loc. cit.}, is given by
\begin{equation}\label{satakeeq}
\cS_{\tL}^{\tM}(\varphi)(m)(\overline{v}) = \sum_{n\in(N_1^-\cap K)^*\backslash N_1^{-,*}}\overline{\varphi(nm)(v)},
\end{equation}
where $\varphi\in \cH_{\tL}(\tV), m\in \tM, v\in \tV$, and the bar indicates the projection $\tV\longtwoheadrightarrow \tV_{(N_1^-\cap K)^*}$.  The map $\cS_{\tL}^{\tM}$ preserves the convolution product, and is therefore an algebra homomorphism.

The necessary properties of the algebras $\cH_{\tL}(\tV)$ and the map $\cS_{\tL}^{\tM}$ are summarized in the following proposition.

\begin{propn}\label{satakeprop}\hfill
\begin{enumerate}
\item Fix $\lambda\in X_*(T)_{L,-}$.  Then the space of functions in $\cH_{\tL}(\tV)$ supported on $(\tL\cap \tK)\widetilde{\lambda}(\varpi)(\tL\cap\tK)$ is one-dimensional.  A generator for this space is given by $T_\lambda^{\tL}$, where $T_\lambda^{\tL}(\widetilde{\lambda}(\varpi))$ is the map 
$$\tV\longtwoheadrightarrow \tV_{(N_{-\lambda}^-\cap K)^*} \stackrel{\sim}{\longleftarrow} \tV^{(N_{-\lambda}\cap K)^*}\longhookrightarrow \tV.$$
This map is a projection.  
\item If $\tL_1\subset\tL_2\subset\tL_3$ are standard Levi subgroups corresponding to standard parabolic subgroups $\tP_1\subset\tP_2\subset\tP_3$, we then have 
$$\cS_{\tL_3}^{\tL_1} = \cS_{\tL_2}^{\tL_1}\circ\cS_{\tL_3}^{\tL_2}.$$
\item Fix an element $\lambda_0\in X_*(T)$ which satisfies $\langle\alpha,\lambda_0\rangle<0$ for $\alpha\in \Pi_L\smallsetminus\Pi_M$ and $\langle\alpha,\lambda_0\rangle = 0$ for $\alpha\in \Pi_M$.  Then the map $\cS_{\tL}^{\tM}:\cH_{\tL}(\tV)\longrightarrow \cH_{\tM}(\tV_{(N_1^-\cap K)^*})$ is a localization at $T_{\lambda_0}^{\tM}$; that is, $\cS_{\tL}^{\tM}$ is injective, $T_{\lambda_0}^{\tM}\in \textnormal{im}(\cS_{\tL}^{\tM})$ is central and invertible, and 
$$\cH_{\tM}(\tV_{(N_1^-\cap K)^*}) = \bigcup_{n\in \bbZ_{\geq 0}}\cS_{\tL}^{\tM}\left(\cH_{\tL}(\tV)\right)\cdot(T_{\lambda_0}^{\tM})^{-n}.$$
\item The algebra $\cH_{\tL}(\tV)$ is commutative.  
\item Assume $\tM = \tT$.  Then the image of the map $\cS_{\tL}^{\tT}$ is supported in the union of cosets of the form $(\tT\cap\tK)\widetilde{\lambda}(\varpi)(\tT\cap\tK)$ for $\lambda\in X_*(T)_{L,-}$.  More precisely, for $\lambda\in X_*(T)_{L,-}$, we have
$$\cS_{\tL}^{\tT}(T_\lambda^{\tL}) = \sum_{\sub{\mu\in X_*(T)_{L,-}}{\mu\geq_L \lambda}} c_{\lambda}(\mu)\tau_\mu,$$
where $c_{\lambda}(\mu)\in \fpb, c_{\lambda}(\lambda) = 1$, and $\leq_L$ is the partial order on $X_*(T)$ defined by $\Pi_L$.  
\item Assume $\tL = \tG$, $\lambda\in X_*(T)_-$, and $\tV = \tF(\nu)$ is a weight of $\tK$ which is $\tM$-regular.  Suppose that either $\Pi_\nu = \Pi_M$, or that $\tV$ is $\tM_{\lambda}$-regular.  We then have
$$\cS_{\tG}^{\tM}(T_\lambda^{\tG}) = T_\lambda^{\tM}.$$
\end{enumerate}
\end{propn}

\begin{proof}
(1)  See \cite{henniartvigneras:modpsatake}, $\S$7.3, Lemma 1.

(2)  See \cite{henniartvigneras:modpsatake}, $\S$2.8, Proposition and \cite{henniartvigneras:cptparabolic}, Proposition 2.3.

(3) See \cite{henniartvigneras:cptparabolic}, Proposition 4.5.

(4)  This follows from part (3), noting that $\cH_{\tT}(\tV_{(U^-\cap L \cap K)^*})$ is commutative.

(5)  See \cite{henniartvigneras:modpsatake}, $\S$7.3, Lemmas 2 and 3, applied to the case of coinvariants (see also \cite{henniartvigneras:cptparabolic}, Proposition 2.3).  

(6)  See \cite{herzig:modpgln}, Corollary 2.18.  
\end{proof}

\begin{notation}
We will denote $T_\lambda^{\tG}$ and $T_{\lambda}^{\tT}$ by $T_\lambda$ and $\tau_{\lambda}$, respectively.  
\end{notation}

\begin{remark}
\label{rem:mixedsat}
We also have a ``mixed'' version of the above construction.  Let $\tV'$ be another weight of $\tL\cap\tK$, and consider the module of intertwiners
$$\cH_{\tL}(\tV,\tV') := \Hom_{\tL}\left(\ind_{\tL\cap\tK}^{\tL}(\tV), \ind_{\tL\cap\tK}^{\tL}(\tV')\right).$$
As with $\cH_{\tL}(\tV)$, we may view elements of $\cH_{\tL}(\tV,\tV')$ as compactly supported, $\tL\cap\tK$-biequivariant, $\Hom_{\fpb}(\tV,\tV')$-valued functions on $\tL$.  Similarly to part (1) of the proposition above, the double coset $(\tL\cap \tK)\widetilde{\lambda}(\varpi)(\tL\cap\tK)$ for $\lambda\in X_*(T)_{L,-}$ supports a nonzero function $\varphi\in \cH_{\tL}(\tV,\tV')$ if and only if $\tV_{(N_{-\lambda}^-\cap K)^*}\cong \tV'_{(N_{-\lambda}^-\cap K)^*}$, in which case $\varphi(\widetilde{\lambda}(\varpi))$ is a scalar multiple of the homomorphism
$$\tV\longtwoheadrightarrow \tV_{(N_{-\lambda}^-\cap K)^*} \cong \tV'_{(N_{-\lambda}^-\cap K)^*}\stackrel{\sim}{\longleftarrow} \tV'^{(N_{-\lambda}\cap K)^*}\longhookrightarrow \tV'.$$

We also have an injective Satake map 
$$\cS_{\tL}^{\tM}:\cH_{\tL}(\tV,\tV')\longrightarrow \cH_{\tM}(\tV_{(N_1^-\cap K)^*},\tV'_{(N_1^-\cap K)^*}),$$
given by the formula \eqref{satakeeq}.   If we take $\tV''$ to be a third weight, we have an obvious product structure
\begin{eqnarray*}
\cH_{\tL}(\tV', \tV'')\times\cH_{\tL}(\tV,\tV') & \longrightarrow & \cH_{\tL}(\tV,\tV''),\\
(\varphi', \varphi) & \longmapsto & \varphi'*\varphi,
\end{eqnarray*}
which is preserved by the Satake map $\cS_{\tL}^{\tM}$.  
\end{remark}

In addition to the above, we have the following results relating parabolic and compact induction.

\begin{propn}\label{compisom}
Let $\tV$ be a weight of $\tK$ and $\tP = \tM\ltimes N^{*}$ a standard parabolic subgroup of $\tG$.  Assume that $\tV$ is $\tM$-regular.  We then have an isomorphism
$$\cH_{\tM}(\tV_{(N^-\cap K)^*})\otimes_{\cH_{\tG}(\tV)}\ind_{\tK}^{\tG}(\tV)\stackrel{\sim}{\longrightarrow}\Ind_{\tP^-}^{\tG}\left(\ind_{\tM\cap \tK}^{\tM}(\tV_{(N^-\cap K)^*})\right),$$
which is both $\tG$- and $\cH_{\tM}(\tV_{(N^-\cap K)^*})$-equivariant.  Here we view $\cH_{\tM}(\tV_{(N^-\cap K)^*})$ as an $\cH_{\tG}(\tV)$-module via $\cS_{\tG}^{\tM}$.
\end{propn}

\begin{proof}
This follows in exactly the same way as Theorem 4.6 (utilizing Corollary 6.5 and Proposition 5.4) of \cite{henniartvigneras:cptparabolic}. 
\end{proof}

Using this proposition, we obtain the following.

\begin{propn}\label{compfactors}
Let $\tV$ be a weight of $\tK$.  Then $\cH_{\tT}(\tV_{(U^-\cap K)^*})\otimes_{\cH_{\tG}(\tV)}\ind_{\tK}^{\tG}(\tV)$ is free as an $\cH_{\tT}(\tV_{(U^-\cap K)^*})$-module.  Moreover, if $\chi:\cH_{\tT}(\tV_{(U^-\cap K)^*})\longrightarrow \fpb$ is a character, then the representation $\chi\circ\cS_{\tG}^{\tT}\otimes_{\cH_{\tG}(\tV)}\ind_{\tK}^{\tG}(\tV)$ has finite length as a $\tG$-module, and its composition factors depend only on $\chi$ and the $(\tT\cap \tK)$-representation $\tV_{(U^-\cap K)^*}$.  
\end{propn}

\begin{proof}
  The proof is almost identical to the proof contained in $\S$4.2 of \cite{abe:irredmodp}.  The only essential difference is in the proof of Lemma 4.13 of \textit{loc.~cit.}: let $\tP_1 = \tM_1\ltimes N_1^{*} \subset \tP_2 = \tM_2\ltimes N_2^{*}$ be two standard parabolic subgroups, and assume $\Pi_{M_2} = \Pi_{M_1}\sqcup\{\alpha_i\}$ for a simple root $\alpha_i$. In analogy with Abe's definition, we take the maps $\Phi_{P_2,P_1}$ and $\Phi_{P_1,P_2}$ to be normalized versions of the maps $\varphi^-$ and $\varphi^+$ from the proof of Theorem \ref{thm:changeofwt} below.  Our normalization differs from that of \textit{loc.~cit.} only when $i = n$: in that case, we let $\Phi_{P_1, P_2} = \varphi^{+}$ and $\Phi_{P_{2}, P_{1} }= \tau_{-2\lambda}\cdot \varphi^{-}$. Our calculation of $\cS_{\tG}^{\tT}(\varphi^{-} * \varphi^{+})$ in Proposition \ref{propn:st2l} then shows that both $\Phi_{P_2,P_1}\circ\Phi_{P_1,P_2}$ and $\Phi_{P_1,P_2}\circ \Phi_{P_2,P_1}$ are identity maps when $i = n$. 
\end{proof}

\subsection{Special features of the metaplectic setting}
\label{twistrep}  
Fix, once and for all, a smooth additive character $\psi:F\longrightarrow \bbC^\times$, and let $\chi_\psi:\tM_{\s}\longrightarrow \fpb^\times$ denote the genuine character defined by
$$\chps(m,\zeta) = \zeta\gamma(\sideset{}{_{M_\s}}\det(m),\psi)^{-1},$$
where $\gamma(a,\psi) = \gamma(\psi_a)/\gamma(\psi)$, $\gamma(\psi)$ is the Weil index, and $\psi_a(x) = \psi(ax)$.  Here we are using the coordinates afforded by Rao's cocycle.  We consider $\chps$ as being valued in $\mu_4(\fpb)$ via a fixed isomorphism $\mu_4(\fpb)\cong \mu_4(\bbC)$.  

Using the character $\chps$, we obtain a bijection between isomorphism classes of mod-$p$ representations of $M_\s\cong\textnormal{GL}_n(F)$ and isomorphism classes of genuine mod-$p$ representations of $\tM_\s$, given by $\tau\longmapsto\tau\otimes\chps$, where we view $\tau$ as a (nongenuine) representation of $\tM_\s$ by inflation from $M_\s$. It follows from (\ref{cocycleMs}) and the construction of the elements $\widetilde{\alpha}_j^\vee(x)$ that we have $\chps(\widetilde{\alpha}_j^\vee(x)) = 1$ for $1\leq j \leq n - 1$.

We now prove a simple lemma which will be of use in determining the support of functions in the image of the Satake map.  Fix $\lambda\in X_*(T)_-, \mu\in X_*(T)$, and $\zeta\in \mu_2$, and define the following sets:
\begin{eqnarray*}
S_{\mu,\lambda} & := & \{u\in (U^-\cap K)\backslash U^-: (U^-\cap K)u\mu(\varpi) \subset K\lambda(\varpi)K\}\\
\tS_{\mu,\lambda,\zeta} & := & \{u\in (U^-\cap K)^*\backslash U^{-,*}: (U^-\cap K)^*u\widetilde{\mu}(\varpi) \subset K^*\widetilde{\lambda}(\varpi)\cdot\zeta K^*\}
\end{eqnarray*}

\begin{lemma} 
\label{setsize}
Let $\lambda\in X_*(T)_-, \mu\in X_*(T)$, and $\zeta\in \mu_2$.  
\begin{enumerate}
\item We have $|S_{\mu,\lambda}| = |\tS_{\mu,\lambda,1}| + |\tS_{\mu,\lambda,-1}|$.  
\item Write $\lambda = \sum_{i = 1}^n a_i\lambda_i, \mu = \sum_{i = 1}^n b_i\lambda_i$, and suppose $\sum_{i = 1}^n (a_i + b_i)$ is odd.  Then we have
$$|\widetilde{S}_{\mu,\lambda,1}| = |\widetilde{S}_{\mu,\lambda,-1}| = |S_{\mu,\lambda}|/2.$$
\end{enumerate}
\end{lemma}

\begin{proof}  Part (1) follows easily using the fact that $\tK\widetilde{\lambda}(\varpi)\tK = K^*\widetilde{\lambda}(\varpi)K^*\sqcup K^*\widetilde{\lambda}(\varpi)\cdot (-1)K^*$, together with the splitting over $U$. For part (2), let $y\in \cO^\times$ and note that the element $\lambda_{n + 1}(y)\in \textnormal{GSp}_{2n}(F)$ normalizes $U$ and $K$.  By uniqueness of the splittings over $U$ and $K$, this implies that $\widetilde{\lambda}_{n + 1}(y)$ normalizes $U^*$ and $K^*$, where $\widetilde{\lambda}_{n + 1}(y)$ denotes an arbitrary lift of $\lambda_{n + 1}(y)$ to $\widetilde{\textnormal{GSp}}_{2n}(F)$.  Suppose $\tS_{\mu,\lambda,\zeta}\neq \emptyset$, and choose $u\in \tS_{\mu,\lambda,\zeta}$.  Conjugating $u$ by $\widetilde{\lambda}_{n + 1}(y)$ and applying equation \eqref{gspconj} gives 
\begin{eqnarray*}
(U^-\cap K)^*(\widetilde{\lambda}_{n + 1}(y)u\widetilde{\lambda}_{n + 1}(y)^{-1})\widetilde{\mu}(\varpi) & = & \widetilde{\lambda}_{n + 1}(y)\left((U^-\cap K)^*u\widetilde{\mu}(\varpi)\right)\widetilde{\lambda}_{n + 1}(y)^{-1}\cdot(\varpi,y)_F^{\sum_{i = 1}^n b_i}\\
& \subset & \widetilde{\lambda}_{n + 1}(y)\left(K^*\widetilde{\lambda}(\varpi)\cdot\zeta K^*\right)\widetilde{\lambda}_{n + 1}(y)^{-1}\cdot(\varpi,y)_F^{\sum_{i = 1}^n b_i}\\
& = & K^*(\widetilde{\lambda}_{n + 1}(y)\widetilde{\lambda}(\varpi)\widetilde{\lambda}_{n + 1}(y)^{-1})\cdot\zeta(\varpi,y)_F^{\sum_{i = 1}^n b_i}K^*\\
& = & K^*\widetilde{\lambda}(\varpi)\cdot\zeta(\varpi,y)_F^{\sum_{i = 1}^n (a_i + b_i)}K^*\\
& = & K^*\widetilde{\lambda}(\varpi)\cdot\zeta(\varpi,y)_FK^*.
\end{eqnarray*}
This shows $\widetilde{\lambda}_{n + 1}(y)u\widetilde{\lambda}_{n + 1}(y)^{-1} \in \widetilde{S}_{\mu,\lambda,\zeta(\varpi,y)_F}$.  In particular, when $y$ is an element of $\cO^{\times}$ whose image in $\kk^\times$ is not a square, conjugation by $\widetilde{\lambda}_{n+1}(y)$ gives an injective map from $\widetilde{S}_{\lambda,\mu,\zeta}$ to $\widetilde{S}_{\lambda,\mu,-\zeta}$.  Combining with part (1) gives the claim.  
\end{proof}

\begin{cor}
\label{cor:oddsum}
Let $\tV$ be a weight of $\tK$, fix $\lambda\in X_*(T)_-$, and write
$$\cS_{\tG}^{\tT}(T_{\lambda}) = \sum_{\sub{\mu\in X_*(T)_-}{\mu\geq \lambda}}c_{\lambda}(\mu)\tau_\mu.$$
Then we have $c_\lambda(\mu) = 0$ for all $\mu$ such that, in the expression of $\mu + \lambda$ as a linear combination of the $\lambda_i$, the coefficients sum to an odd number. 
\end{cor}

\begin{proof}
Fix $y\in \cO^\times$ whose image in $\mathfrak{k}^\times$ is not a square, and let $\delta\in \widetilde{\textnormal{GSp}}_{2n}(F)$ denote some fixed lift of $\lambda_{n + 1}(y)$.  It is known (cf. \cite{herzig:modpgln}, Proposition 2.2) that for every weight $V$ of $K$, there exists a weight of $\textnormal{GSp}_{2n}(\cO)$ whose restriction to $K$ is $V$.  Therefore, we may endow $\tV$ with an action of the element $\delta$.  In particular, since $\delta$ normalizes $(N_{-\lambda}\cap K)^*$, the explicit description of the operator $T_\lambda$ shows that $\delta T_\lambda(\widetilde{\lambda}(\varpi))\delta^{-1} = T_\lambda(\widetilde{\lambda}(\varpi))$.

For $\mu$ and $\lambda$ as in the statement of the corollary (with the coefficients of $\mu + \lambda$ summing to an odd number) and $u \in \tS_{\mu,\lambda,1}$, there exist $k_{u}^{*}, k_u'^* \in K^*$ such that $u\widetilde{\mu}(\varpi) = k_u^*\widetilde{\lambda}(\varpi)k_u'^*$.  Then, since $\delta$ is a lift of $\lambda_{n + 1}(y)$, we proceed as in the proof of Lemma \ref{setsize}(2) to obtain 
\begin{eqnarray*}
(\delta u \delta^{-1})\widetilde{\mu}(\varpi) & = & (\varpi,y)_F\cdot (\delta k_u^* \delta^{-1})\widetilde{\lambda}(\varpi)(\delta k_u'^* \delta^{-1})\\
 & = & (-1)\cdot (\delta k_u^* \delta^{-1})\widetilde{\lambda}(\varpi)(\delta k_u'^* \delta^{-1}).
 \end{eqnarray*}
Fix a nonzero $v\in \tV^{(U\cap K)^*}\stackrel{\sim}{\longrightarrow}\tV_{(U^-\cap K)^*}$, and let $c$ denote the eigenvalue of $\delta$ on $v$.  We then have
\begin{eqnarray*}
\overline{T_{\lambda}(u\widetilde{\mu}(\varpi))(v)} & = & \overline{k_u^*\circ T_{\lambda}(\widetilde{\lambda}(\varpi))\circ k_u'^*.v }\\
& = & c^{-1}\delta\cdot\left(\overline{k_u^*\circ T_{\lambda}(\widetilde{\lambda}(\varpi))\circ k_u'^*\delta^{-1}\cdot(cv) }\right)\\
& = & \overline{(\delta k_u^*\delta^{-1})\circ (\delta T_{\lambda}(\widetilde{\lambda}(\varpi))\delta^{-1})\circ (\delta k_u'^*\delta^{-1})\cdot v }\\
& = & \overline{(\delta k_u^*\delta^{-1})\circ T_{\lambda}(\widetilde{\lambda}(\varpi))\circ (\delta k_u'^*\delta^{-1})\cdot v }\\
& = & \overline{T_{\lambda}((\delta k_u^*\delta^{-1})\widetilde{\lambda}(\varpi)(\delta k_u'^*\delta^{-1}))(v)}\\
& = & -\overline{T_{\lambda}((\delta u\delta^{-1})\widetilde{\mu}(\varpi))(v)}.
\end{eqnarray*}
Hence, using Lemma \ref{setsize}, we obtain
\begin{eqnarray*}
c_{\lambda}(\mu)\overline{v} & = & \cS_{\tG}^{\tT}(T_{\lambda})(\widetilde{\mu}(\varpi))(\overline{v})\\
& = & \sum_{u\in \tS_{\mu,\lambda,1}}\overline{T_\lambda(u\widetilde{\mu}(\varpi))(v)} + \sum_{u\in \tS_{\mu,\lambda,-1}}\overline{T_\lambda(u\widetilde{\mu}(\varpi))(v)}\\
& = & \sum_{u\in \tS_{\mu,\lambda,1}}\overline{T_\lambda(u\widetilde{\mu}(\varpi))(v)} + \sum_{u\in \tS_{\mu,\lambda,1}}\overline{T_\lambda((\delta u\delta^{-1})\widetilde{\mu}(\varpi))(v)}\\
& = & 0.
\end{eqnarray*}
\end{proof}

In the following proposition, we let unaccented symbols correspond to the appropriate analogous objects for the reductive group $G$ (so $\cS_{L}^T$ denotes the Satake map with respect to a weight $V$, etc.).

\begin{propn}
\label{samecoeffs}
Let $L$ be a standard Levi subgroup of $G$ such that $L\subset M_\s$, and let $1$ denote the trivial weight of $L$.  Fix $\lambda\in X_*(T)_{L,-}$, and write 
$$\cS_L^T(T_\lambda^L) = \sum_{\sub{\mu\in X_*(T)_{L,-}}{\mu\geq_L \lambda}}c_\lambda(\mu)T_\mu^T.$$
Let $\cS_{\tL}^{\tT}$ denote the Satake map for the weight $1\boxtimes\varepsilon$.  We then have
$$\cS_{\tL}^{\tT}(T_\lambda^{\tL}) = \sum_{\sub{\mu\in X_*(T)_{L,-}}{\mu\geq_L \lambda}}\widetilde{c}_\lambda(\mu)\tau_\mu,$$
where $\widetilde{c}_\lambda(\mu) = c_\lambda(\mu)$.
\end{propn}

\begin{proof}
Let $N_1 = U\cap L$, so that $T\ltimes N_1$ is a parabolic subgroup of $L$.  For $\lambda,\mu\in X_*(T)_{L,-}$ and $\zeta\in \mu_2$, we set
\begin{eqnarray*}
S_{\mu,\lambda}' & := & \{n\in (N_1^-\cap K)\backslash N_1^- : (N_1^-\cap K)n\mu(\varpi)\subset (L\cap K)\lambda(\varpi)(L\cap K)\},\\
\tS_{\mu,\lambda,\zeta}' & := & \{n\in (N_1^-\cap K)^*\backslash N_1^{-,*} : (N_1^-\cap K)^*n\widetilde{\mu}(\varpi)\subset (L\cap K)^*\widetilde{\lambda}(\varpi)\cdot\zeta(L\cap K)^*\}.
\end{eqnarray*}
By definition, the set $\tS_{\mu,\lambda,1}'\sqcup\tS_{\mu,\lambda,-1}'$ is exactly the image of $S_{\mu,\lambda}'$ under the splitting over $N_1^-$, and therefore the splitting gives a bijection $S_{\mu,\lambda}'\stackrel{\sim}{\longrightarrow} \tS_{\mu,\lambda,1}'\sqcup\tS_{\mu,\lambda,-1}'$.  Since we are considering the trivial weight of $K$ and the weight $1 \boxtimes \varepsilon$ of $\widetilde{K}$, the definition of the Satake map shows that
$$c_\lambda(\mu) = |S_{\mu,\lambda}'|,\qquad \widetilde{c}_\lambda(\mu) = |\tS_{\mu,\lambda,1}'| - |\tS_{\mu,\lambda,-1}'|$$
(note that $1 \boxtimes \varepsilon$ takes the value $\varepsilon(\zeta)$ on $(1, \zeta)K^*$, and the coefficient $\widetilde{c}_{\lambda}(\mu)$ must account for both $\zeta = 1$ and $\zeta = -1$).

Let us fix $\lambda,\mu\in X_*(T)_{L,-}$.  First, we claim that there exists $\zeta\in \mu_2$ such that $|S_{\mu,\lambda}'| = |\tS_{\mu,\lambda,\zeta}'|$.  To see this, suppose that $n\in S_{\mu,\lambda}$, so that
$$n\mu(\varpi) = k_n\lambda(\varpi)k_n'$$
for some $k_n, k_n'\in L\cap K$.  Comparing determinants of both sides shows that $\det_{M_\s}(k_nk_n') = 1$.  Since $M_\s^{\textnormal{der}}\cong\textnormal{SL}_n(F)$, we see that by adjusting $k_n$ and $k_n'$ if necessary, we may assume without loss of generality that $k_n, k_n'\in L\cap K \cap M_\s^{\textnormal{der}}$.

Consider now the subgroup $M_\s^{\textnormal{der}}$.  Examining the pullback to $M_\s^{\textnormal{der}}$ of the Steinberg cocycle of $\tG$ (cf. \cite{moore:gpextensions}, Chapter III; in particular Theorem 8.1), we see that the cover $\tG\longrightarrow G$ splits (uniquely) over $M_\s^{\textnormal{der}}$.  In the coordinates of Rao's cocycle, this splitting is given by
$$m\longmapsto m^* = (m,1),$$
where $m\in M_\s^{\textnormal{der}}$.  
This has the following consequence.  The splitting over $N_1^-$ gives a bijection between $S_{\mu,\lambda}'$ and $\tS_{\mu,\lambda,1}'\sqcup\tS_{\mu,\lambda,-1}'$, and we would like to determine in which signed set $n^*$ lies, where $n$ is as above.  Let us write
$$\widetilde{\mu}(\varpi) = (\mu(\varpi),\zeta_\mu),\qquad \widetilde{\lambda}(\varpi) = (\lambda(\varpi),\zeta_\lambda),$$
with $\zeta_\mu,\zeta_\lambda\in \mu_2$.  Using equation \eqref{cocycleMs}, we obtain
\begin{eqnarray*}
n^*\widetilde{\mu}(\varpi) & = & (n,1)\cdot(\mu(\varpi),\zeta_\mu)\\
& = & (n\mu(\varpi),\zeta_\mu),\\
k_n^*\widetilde{\lambda}(\varpi)\cdot\zeta k_n'^* & = & (k_n,1)\cdot(\lambda(\varpi),\zeta_\lambda)\cdot(1,\zeta)\cdot(k_n',1)\\
& = & (k_n\lambda(\varpi)k_n',\zeta\zeta_\lambda).
\end{eqnarray*}
Therefore, $n^*$ lies in $\tS_{\mu,\lambda,\zeta}$, where $\zeta = \zeta_\mu\zeta_\lambda^{-1}$.  Since $\zeta$ does not depend on $n$, the claim follows.

Now, if $1\leq i \leq n - 1$, we have $\alpha_i^\vee(\varpi)\in M_\s^{\textnormal{der}}$, and therefore equation \eqref{cocycleMs} and the definition of $\widetilde{\alpha}_i^\vee(\varpi)$ imply $\widetilde{\alpha}_i^\vee(\varpi) = (\alpha_i^\vee(\varpi),1)$.  Thus
\begin{eqnarray*}
\left((\lambda + \alpha_i^\vee)(\varpi),\zeta_{\lambda + \alpha_i^\vee}\right) & = & \widetilde{(\lambda + \alpha_i^\vee)}(\varpi)\\
 & = & \widetilde{\lambda}(\varpi)\cdot\widetilde{\alpha}_i^\vee(\varpi)\\
 & = & (\lambda(\varpi),\zeta_\lambda)\cdot(\alpha_i^\vee(\varpi),1)\\
 & = & \left((\lambda+ \alpha_i^\vee)(\varpi),\zeta_{\lambda}\right).
\end{eqnarray*}
Since $\mu\geq_L\lambda$, we obtain $\zeta_\mu = \zeta_\lambda$, and therefore $\zeta = 1$.  
\end{proof}

\section{Change of weight}\label{ch:changeofwt}

We now discuss the change-of-weight theorem, which will be the main technical input to our classification results in $\S$\ref{classnchapter}.

\subsection{Statement and proof}
\label{subsec:chwt1}

Fix $1 \leq i \leq n$ and let $\nu \in X_q(T)$ be a $q$-restricted weight such that $\langle \nu, \alpha_i^{\vee}\rangle = 0$. Let $\omega_{\alpha_i}$ be the fundamental weight associated to $\alpha_i$, and set $\nu' := \nu + (q-1)\omega_{\alpha_i}$.  Note that $\nu'$ is also a $q$-restricted weight.  Let $\tV = \tF(\nu)$ (resp., $\tV' = \tF(\nu')$) denote the associated weight of $\tK$. Then $\tV_{(U^- \cap K)^*} \cong \tV'_{(U^- \cap K)^*}$ as genuine representations of $(\tT \cap \tK)$, and via the Satake transform we identify the algebras $\cH_{\tG}(\tV)$ and $\cH_{\tG}(\tV')$.

Given a simple root $\alpha\in \Pi$, let $\lambda_{\alpha} \in X_*(T)$ denote a cocharacter such that $\langle\alpha, \lambda_{\alpha}\rangle < 0$ and $\langle\Pi\smallsetminus\{\alpha\}, \lambda_\alpha\rangle = 0$.  Let $\chi:\cH_{\tG}(\tV)\longrightarrow \fpb$ be a character, and set
\begin{equation}\label{defofpichi}
\Pi(\chi):=\left\{\alpha\in \Pi: (\chi\circ(\cS_{\tG}^{\tT})^{-1})(\tau_{\lambda_\alpha}) = 0\right\}
\end{equation}
(note that this definition is independent of the choice of $\lambda_\alpha$).  The character $\chi$ then factors as $\chi = \chi'\circ \cS_{\tG}^{\tM_{\Pi(\chi)}}$, where $\chi':\cH_{\tM_{\Pi(\chi)}}(\tV_{(N_{\Pi(\chi)}^-\cap K)^*})\longrightarrow \fpb$ is some character; this follows from the fact that $\cS_{\tG}^{\tM_{\Pi(\chi)}}$ is a localization at $T_{\lambda'}^{\tM_{\Pi(\chi)}}$, where $\lambda' = \sum_{\alpha\not\in\Pi(\chi)}\lambda_\alpha$. The main theorem of this chapter is the following.

\begin{thm}\label{thm:changeofwt}
Fix a simple root $\alpha_i \in \Pi$, let $\chi: \cH_{\tG}(\tV) \longrightarrow \fpb$ be a character, and let $\Pi(\chi)$ be as above.  Assume that $\alpha_i \notin \Pi(\chi)$. If $i \neq n$, further assume that either $\langle \Pi(\chi), \alpha_i^{\vee}\rangle \neq 0$, or that $\langle\Pi(\chi), \alpha_i^{\vee}\rangle = 0$ and $(\chi'\circ(\cS_{\tM_{\Pi(\chi)}}^{\tT})^{-1})(\tau_{\alpha_i^\vee})\neq 1$ (note that this is well-defined as $\alpha_i^\vee\in X_*(T)_{M_{\Pi(\chi)},-}$). We then have
$$\chi\otimes_{\cH_{\tG}(\tV)}\ind_{\tK}^{\tG}(\tV)\cong \chi\otimes_{\cH_{\tG}(\tV')}\ind_{\tK}^{\tG}(\tV').$$
\end{thm}

\begin{proof}
Let $\lambda := -\sum_{j = 1}^i \lambda_j$; then 
$$\langle \alpha_j, -\lambda \rangle = 0~ \textnormal{if}~ j \neq i,~ \textnormal{and}~ \langle\alpha_i, -\lambda \rangle = \begin{cases}
1 & \textnormal{ if } 1 \leq i \leq n-1,\\
2 & \textnormal{ if } i = n.
\end{cases}$$
Note that $\lambda$ is a possible choice of $\lambda_{\alpha_i}$ in the statement of the theorem. Furthermore, $-\lambda$ is a minuscule fundamental coweight associated to $\alpha_i$ if $1 \leq i \leq n-1$, and for all $1 \leq i \leq n$ the inner product $\langle \alpha_i, -\lambda \rangle$ is as small as possible under the constraints on $\lambda_{\alpha_i}$.

Lemma \ref{weights} and the definition of the group $N_\lambda$ imply that $\tV_{(N_{\lambda}\cap K)^*} \cong \tV'_{(N_{\lambda}\cap K)^*}$.  Therefore, by Remark \ref{rem:mixedsat} following Proposition \ref{satakeprop}, there exist nonzero elements $\varphi^+\in \cH_{\tG}(\tV,\tV')$ and $\varphi^-\in \cH_{\tG}(\tV',\tV)$ which are supported in $\tK\widetilde{\lambda}(\varpi)\tK$.

\begin{sublemma}
\label{sl1}
The function $\varphi^- * \varphi^+\in \cH_{\tG}(\tV)$ is nonzero and has support $\tK\widetilde{\lambda}(\varpi)^2\tK$; that is, $\varphi^- * \varphi^+ = cT_{2\lambda}$ for some nonzero constant $c$.  
\end{sublemma}

\begin{proof}[Proof of Sublemma \ref{sl1}]
See the proof of Lemma 4.3 in \cite{abe:irredmodp}, noting that conjugation by elements of $\tT$ preserves $U^*$ by uniqueness of the splitting over $U$.
\end{proof}

After possible rescaling, we may assume $\varphi^- * \varphi^+ = T_{2\lambda}$.  The maps $\varphi^-,\varphi^+$ induce maps $\boldsymbol{\varphi}^-, \boldsymbol{\varphi}^+$ as follows:

\centerline{
\xymatrix{ \chi\otimes_{\cH_{\tG}(\tV)}\ind_{\tK}^{\tG}(\tV) \ar@<1ex>[r]^{\boldsymbol{\varphi}^+} & \chi\otimes_{\cH_{\tG}(\tV')}\ind_{\tK}^{\tG}(\tV') \ar@<1ex>[l]^{\boldsymbol{\varphi}^-}
}
}
\noindent and the composition of $\boldsymbol{\varphi}^-$ with $ \boldsymbol{\varphi}^+$, in either order, is the scalar $\chi (\varphi^-*\varphi^+)$ (note that $\boldsymbol{\varphi}^-\circ\boldsymbol{\varphi}^+ = \boldsymbol{\varphi}^+\circ\boldsymbol{\varphi}^-$ under the identification $\cH_{\tG}(\tV) \cong \cH_{\tG}(\tV')$; see \cite{herzig:notes}, Proposition 31(3)).  Proposition \ref{propn:st2l} below gives
$$\cS_{\tG}^{\tT}(\varphi^{-}*\varphi^+) =\begin{cases}
 \tau_{2\lambda} - \tau_{2\lambda + \alpha_i^{\vee}} & \textnormal{if } 1 \leq i \leq n-1,\\
 \tau_{2\lambda} & \textnormal{if } i = n,
\end{cases}$$
so that
$$\chi (\varphi^- * \varphi^+) = \begin{cases} (\chi\circ(\cS_{\tG}^{\tT})^{-1})(\tau_{2\lambda} - \tau_{2\lambda + \alpha_i^\vee}) & \textnormal{if}~ 1\leq i \leq n - 1, \\ (\chi\circ(\cS_{\tG}^{\tT})^{-1})(\tau_{2\lambda}) & \textnormal{if}~i = n.\end{cases}$$
Our assumptions now guarantee that $\chi(\varphi^- * \varphi^+)\neq 0$, and therefore we obtain our desired isomorphism $$\chi\otimes_{\cH_{\tG}(\tV)}\ind_{\tK}^{\tG}(\tV)\stackrel{\sim}{\longrightarrow} \chi\otimes_{\cH_{\tG}(\tV')}\ind_{\tK}^{\tG}(\tV').$$
\end{proof}

We now carry out the calculation of $\cS_{\tG}^{\tT}(T_{2\lambda})$ which was used in the proof of Theorem \ref{thm:changeofwt}.

\begin{propn}
\label{propn:st2l}
Let $\alpha_i$, $\tV$, $\tV'$, $\lambda$ be as in the proof of Theorem \ref{thm:changeofwt}, and let $\mathcal{S}_{\tG}^{\tT}$ denote the Satake transform $\cH_{\tG}(\tV) \longrightarrow \cH_{\tT}(\tV_{(U^{-} \cap K)^*})$. Then 
$$\mathcal{S}_{\tG}^{\tT}(T_{2\lambda}) = \begin{cases} \tau_{2\lambda} - \tau_{2\lambda + \alpha_{i}^{\vee}} & \textnormal{ if } 1 \leq i \leq n-1,\\
\tau_{2\lambda} & \textnormal{ if } i = n. \end{cases}$$
\end{propn}

\begin{proof}
Let $\Pi_{\nu}$ be as defined in $\S$\ref{repns}, so that $\alpha_i\in \Pi_\nu$, and let $P_{\nu} = M_{\nu} \ltimes N_{\nu}$ be the parabolic subgroup of $G$ corresponding to $\Pi_{\nu}$. By the transitivity of Satake morphisms (Proposition \ref{satakeprop}(2)), we have
\begin{equation}
\label{eqn:sfactor}
\cS^{\tT}_{\tG}(T_{2\lambda}) = \cS_{\tM_\s \cap \tM_{\nu}}^{\tT} \circ \cS_{\tM_{\nu}}^{\tM_\s \cap \tM_{\nu}} \circ \cS_{\tG}^{\tM_{\nu}}(T_{2\lambda}).
\end{equation}
We will compute each of these maps in sequence.  Note first that, by definition of $M_{\nu}$, the weight $\tV = \tF(\nu)$ is $\tM_\nu$-regular, and Proposition \ref{satakeprop}(6) implies that $\mathcal{S}_{\tG}^{\tM_{\nu}}(T_{2\lambda}) = T^{\tM_{\nu}}_{2\lambda} \in \cH_{\tM_{\nu}}(\tV_{(N^{-}_{\nu} \cap K)^*})$.

Since $\langle \nu, \alpha^{\vee}\rangle = 0$ for all $\alpha \in \Pi_{\nu}$, the representation $V_{N^{-}_{\nu} \cap K} \cong F_{M_{\nu}}(\nu)$ of $M_{\nu} \cap K$ is one-dimensional; for brevity, we denote this one-dimensional representation by $\nu$.  By Corollary 3.4 of \cite{abe:irredmodp}, there exists a unique character $\nu_{M_\nu}$ of $M_{\nu}$ such that $\nu_{M_{\nu}}|_{(M_{\nu}\cap K)} = \nu$ and $\nu_{M_{\nu}}(\mu(\varpi)) = 1$ for all $\mu \in X_{*}(T)$.  We let $\nu_{M_\nu}^*$ denote the nongenuine character of $\tM_\nu$ given by $\nu_{M_\nu}\circ \pr$.  The following sublemma will allow us to assume that $\tV_{(N_{\nu}^{-} \cap K)^*} \cong 1 \boxtimes \varepsilon$ as a representation of $\tM_{\nu} \cap \tK$ in the calculation of $\cS_{\tM_{\nu}}^{\tM_{\nu} \cap \tM_{\s}}(T^{\tM_{\nu}}_{2\lambda})$, where $1$ denotes the trivial representation of $M_\nu\cap K$.

\begin{sublemma}
\label{sl2}
Maintain the notation of the previous paragraph, and let $M\subset M_\nu$ be a standard Levi subgroup. Then:
\begin{enumerate} 
\item the map $s_{\nu}: \cH_{\tM_{\nu}}(\nu\boxtimes\varepsilon) \longrightarrow \cH_{\tM_{\nu}}(1 \boxtimes \varepsilon)$ given by $\varphi \mapsto (\nu_{M_{\nu}}^*)^{-1} \cdot\varphi$ is an algebra isomorphism;
\item the map $\varphi \mapsto (\nu_{M_{\nu}}^*)^{-1}|_{\tM}\cdot\varphi$ defines an algebra isomorphism $\cH_{\tM}(\nu\boxtimes\varepsilon) \longrightarrow \cH_{\tM}(1 \boxtimes \varepsilon)$, also denoted by $s_{\nu}$, and the following diagram is commutative:
\begin{displaymath}
\xymatrix{\cH_{\tM_{\nu}}(\nu\boxtimes\varepsilon) \ar[r]^{s_{\nu}} \ar[d]_{\cS^{\tM}_{\tM_{\nu}}} & \cH_{\tM_{\nu}}(1 \boxtimes \varepsilon) \ar[d]^{\cS^{\tM}_{\tM_{\nu}} }\\ \cH_{\tM}(\nu\boxtimes\varepsilon) \ar[r]^{s_{\nu}} & \cH_{\tM}(1 \boxtimes \varepsilon)}
\end{displaymath}
where the downward arrows are the partial Satake transforms $\cH_{\tM_{\nu}}(\nu\boxtimes\varepsilon) \longrightarrow \cH_{\tM}(\nu\boxtimes\varepsilon)$ and\\ \noindent $\cH_{\tM_{\nu}}(1\boxtimes \varepsilon) \longrightarrow \cH_{\tM}(1 \boxtimes \varepsilon)$, respectively. 
\end{enumerate}
\end{sublemma}

\begin{proof}[Proof of Sublemma \ref{sl2}]
See \cite{abe:irredmodp} Lemma 3.5 and \cite{herzig:modpgln} Lemma 4.6. 
\end{proof}

In the remainder of the proof, we assume $\tV_{(N_{\lambda} \cap K)^*} \cong 1 \boxtimes \varepsilon$. Suppose that $\alpha_n \not\in \Pi_{\nu}$. Then the Satake map $\cS_{\tM_\nu}^{\tM_\nu\cap \tM_\s}$ is trivial, so it only remains to calculate $\cS_{\tM_{\nu} \cap \tM_\s}^{\tT}(T^{\tM_{\nu} \cap \tM_{\s}}_{2\lambda})$. The assumption that $\alpha_n \notin \Pi_{\nu}$ implies that $i \neq n$, and therefore Proposition 6.7 of \cite{herzig:modpgln} and Proposition \ref{samecoeffs} imply 
$$\cS_{\tM_{\nu} \cap \tM_{\s}}^{\tT}(T^{\tM_{\nu} \cap \tM_{\s}}_{2\lambda}) = \cS_{\tG}^{\tT}(T_{2\lambda}) = \tau_{2\lambda} - \tau_{2\lambda + \alpha_i^\vee}.$$

From now on we suppose that $\alpha_n \in \Pi_{\nu}$.  Define the unipotent subgroup $N_1\subset M_{\nu}$ by $N_1:=M_\nu\cap N_\s$, so that $(\tM_\nu\cap \tM_\s)\ltimes N_1^*$ is a standard parabolic subgroup of $\tM_\nu$.  By definition of the partial Satake transform applied to the Hecke algebra of the one-dimensional weight $1 \boxtimes \varepsilon$, we have 
\begin{equation}
\label{redsateqn}
\cS^{\tM_{\nu} \cap \tM_{\s}}_{\tM_{\nu}}(T^{\tM_{\nu}}_{2\lambda})(\widetilde{\mu}(\varpi)) = \sum_{n \in(N_1^{-}\cap K)^* \backslash N_1^{-, *}}T^{\tM_{\nu}}_{2\lambda}(n\widetilde{\mu}(\varpi))
\end{equation}
for each $\mu \in X_*(T)_{M_{\nu} \cap M_{\s}, -}$ (we omit the vector $v$ here since $1\boxtimes \varepsilon$ is one-dimensional).  For such $\mu$ and for $\zeta \in \mu_2$, we set 
$$\tS''_{\mu, 2\lambda, \zeta} := \{ n \in (N_1^{-} \cap K)^* \backslash N_1^{-, *}: \, (N_1^{-} \cap K)^*n\widetilde{\mu}(\varpi) \subset (M_\nu\cap K)^*\widetilde{\lambda}(\varpi)^2 \cdot \zeta (M_\nu\cap K)^*\}.$$ 
Then the right-hand side of \eqref{redsateqn} is equal to $|\tS''_{\mu, 2\lambda, 1}| - |\tS''_{\mu, 2\lambda, -1}|$.

Let $J$ denote the pro-$p$ subgroup of $K \cap M_{\nu} \cap M_{\s}$ which is generated by $\{ u_{\alpha}(x): \, \alpha \in \Pi_{\nu}\cap \Pi_\s, \, x \in \cO\}$ and $\{\eta(y):\eta\in X_*(T), y\in 1 + \varpi\cO\}$, and let $J^*$ denote the image of $J$ in $K^*$ under the splitting over $K$.  In what follows, we let $\Phi_{\nu}$ and $\Phi_{\s}$ denote the root systems generated by $\Pi_{\nu}$ and by $\Pi_{\s}$, respectively, and define the subset of positive roots in each system accordingly. Write $(\Phi_{\nu} \cap \Phi_{\s})^{+} = \Phi_{\nu}^{+} \cap \Phi_{\s}^{+}$.

\begin{sublemma}
\label{sl3}
Suppose that $\alpha_n \in \Pi_{\nu}$, and let $\mu \in X_{*}(T)_{M_{\nu} \cap M_{\s}, -}$ and $\zeta \in \mu_2$. Then $J^*$ acts by conjugation on $\tS''_{\mu, 2\lambda, \zeta}$.
\end{sublemma}

\begin{proof}[Proof of Sublemma \ref{sl3}]

We first claim that $J^*$ normalizes $(N_1^{-}\cap K)^*$. Since both subgroups belong to $K^*$ and $J$ is a subgroup of $K \cap M_{\nu} \cap M_{\s}$, it is enough to show that $K \cap M_\nu \cap M_\s$ normalizes $(N_1^-\cap K)$.  Since $M_\s$ normalizes $N_\s^-$, we have that $M_\s \cap M_\nu \cap K$ indeed normalizes $N_\s^- \cap M_\nu \cap K = N_1^- \cap K$.

Next let $\widetilde{j} \in J^*$.  We claim that $\widetilde{\mu}(\varpi)^{-1}\widetilde{j}\widetilde{\mu}(\varpi)\in J^*$.  Since $\tT$ is abelian, it suffices to consider the case when $\widetilde{j} = \widetilde{u}_\alpha(x)$ for $\alpha\in \Pi_\nu\cap \Pi_\s,  x\in \cO$.  Since each $U_{\alpha}$ is a unipotent subgroup of $G$ and consequently admits a unique splitting, we have 
$$\widetilde{\mu}(\varpi)^{-1}\widetilde{u}_\alpha(x)\widetilde{\mu}(\varpi) = \widetilde{u}_{\alpha}(\varpi^{-\langle\alpha,\mu\rangle}x)\in J^*.$$

Now let $n \in \tS''_{\mu, 2\lambda, \zeta}$, so that $(N_1^{-} \cap K)^* n\widetilde{\mu}(\varpi) \subset (M_\nu\cap K)^* \widetilde{\lambda}(\varpi)^2\cdot \zeta (M_\nu\cap K)^*$, and take $\widetilde{j} \in J^*$.  Then, applying the above two comments we obtain
\begin{eqnarray*}
(N_1^{-} \cap K )^* \widetilde{j}n\widetilde{j}^{-1}\widetilde{\mu}(\varpi) & = & \widetilde{j} (N_1^{-} \cap K)^*n \widetilde{\mu}(\varpi)\widetilde{\mu}(\varpi)^{-1} \widetilde{j}^{-1} \widetilde{\mu}(\varpi)\\
& \subset & J^*  \left((N_1^{-} \cap K)^* n\widetilde{\mu}(\varpi)\right)J^*\\
& \subset & J^* (M_\nu\cap K)^*\widetilde{\lambda}(\varpi)^2\cdot \zeta (M_\nu\cap K)^*J^*\\
& = & (M_\nu\cap K)^* \widetilde{\lambda}(\varpi)^2\cdot \zeta (M_\nu\cap K)^*.
\end{eqnarray*}
Hence $\widetilde{j}n\widetilde{j}^{-1} \in \tS_{\mu, 2\lambda, \zeta}''$ as well. 
\end{proof}

Since $J^*$ is a pro-$p$ group, the orbits of its action on $\tS''_{\mu, 2\lambda, \zeta}$ each have size equal to a power of $p$. Therefore $|\tS''_{\mu, 2\lambda, \zeta}| \not \equiv 0 \pmod{p}$ only if the action of $J^*$ on $\tS''_{\mu, 2\lambda, \zeta}$ has a fixed point. The following facts will be useful in determining when this occurs.

\begin{sublemma}
\label{sl3.5}
\hfill
\begin{enumerate}
\item Let $\gamma \in \Phi_{\nu}^+ \smallsetminus (\Phi_{\nu} \cap \Phi_{\s})^{+}$. Then there exists $\beta \in (\Phi_{\nu} \cap \Phi_{\s})^{+}$ such that $\beta - \gamma \in \Phi$ if and only if $\gamma \neq \alpha_n$. If $\beta \in (\Phi_{\nu} \cap \Phi_{\s})^{+}$ and $\beta - \gamma \in \Phi$, then
$$[u_{\beta}(1), u_{-\gamma}(x)] = \prod_{j = 1}^{\ell} u_{j\beta -\gamma}(c_{\beta, -\gamma; j, 1}\, x)$$
with $1 \leq \ell \leq 2$, $-(j\beta - \gamma) \in \Phi_{\nu}^{+}\smallsetminus (\Phi_{\nu} \cap \Phi_{\s})^{+}$ and $c_{\beta, -\gamma; j, 1} \in \cO^\times$ for $1 \leq j \leq \ell$. If $\beta - \gamma \notin \Phi$, then $[u_{\beta}(1), u_{-\gamma}(x)] = 1$ for all $x \in F$. 
\item Let $n \in N_1^{-}$. Then $[u_{\beta}(1), n] \in (N_{1}^{-} \cap K)$ for all $\beta \in (\Phi_{\nu} \cap \Phi_{\s})^{+}$ if and only if $(N_1^{-} \cap K)n = (N_{1}^{-} \cap K)u_{-\alpha_n}(x)$ for some $x \in F$. 
\end{enumerate}
\end{sublemma}

\begin{proof}[Proof of Sublemma \ref{sl3.5}]
(1)  The first statement follows from an examination of the possible coefficients in the expression of any root as a linear combination of simple roots. Similar considerations show that if  $j \beta - k \gamma \in \Phi$ with $\gamma \in \Phi_{\nu}^{+} \smallsetminus (\Phi_{\nu} \cap \Phi_{\s})^{+}$, $\beta \in (\Phi_{\nu} \cap \Phi_{\s})^{+}$, and $j, k > 0$, then $k = 1$ and $1 \leq j \leq 2$. If $\beta - \gamma \in \Phi$, then $-(\beta - \gamma) \in \Phi_{\nu}^{+} \smallsetminus (\Phi_{\nu} \cap \Phi_{\s})^{+}$ and the $\beta$-string of roots through $-\gamma$ is either (i) $\{ -\gamma, \beta - \gamma\}$, (ii) $\{ -\gamma, \beta - \gamma, 2\beta - \gamma\}$, or (iii) $\{-\beta - \gamma, -\gamma, \beta - \gamma\}$.  Lemma 15 and Lemma 2(b) of \cite{steinberg:chevalleygps} imply that the structure constants $c_{ \beta, -\gamma; 1, 1}$ appearing in the commutator formula (\ref{ucomm}) take the following values:  $c_{\beta, -\gamma; 1, 1} \in \{\pm 1\}$ in cases (i) and (ii), and $c_{\beta, -\gamma; 1, 1} \in\{ \pm 2\}$ in case (iii). We have $\ell = 1$ in cases (i) and (iii), while in case (ii) we have $\ell = 2$ and $c_{\beta, -\gamma; 2, 1} \in \{\pm 1\}$ (this last value follows from Lemma 9.2.2 of \cite{springer:linalggrps} together with our assumption that all structure constants lie in $\bbZ$). Since $\text{char}(\mathfrak{k}) \neq 2$, we thus have $c_{\beta, -\gamma; j, 1} \in \cO^\times$ for all $1 \leq j \leq \ell$. When $\beta - \gamma \notin \Phi$, we have $j\beta - k\gamma \notin \Phi$ for all $j, k > 0$ and so (\ref{ucomm}) implies that $[u_{\beta}(1), u_{-\gamma}(x)] = 1$ for all $x \in F$.

(2)  For each $\beta \in (\Phi_{\nu} \cap \Phi_{\s})^{+}$, the element $u_{\beta}(1)$ normalizes $(N_1^{-} \cap K)$ (by the proof of Sublemma \ref{sl3}) and centralizes the root subgroup $U_{-\alpha_n}$ (by part (1)), so $[u_{\beta}(1), n] \in (N_1^{-} \cap K)$ if $(N_1^{-} \cap K)n = (N_1^{-} \cap K)u_{-\alpha_n}(x)$ for some $x \in F$. Conversely, suppose $[u_\beta(1),n]\in (N_1^-\cap K)$ for some fixed $\beta \in (\Phi_\nu\cap \Phi_\s)^+$.  Since $-j\gamma - k\gamma' \notin \Phi$ for all $j, k > 0$ and all $\gamma, \gamma' \in \Phi_{\nu}^{+}\smallsetminus (\Phi_{\nu} \cap \Phi_{\s})^{+}$, it follows from (\ref{ucomm}) that $N_1^{-}$ is abelian.  We may therefore assume that $n = \prod_{a = 1}^m u_{-\gamma_a}(x_a)$, where $\{\gamma_a : 1 \leq a \leq m\}$ is a set of distinct elements of $\Phi_{\nu}^{+} \smallsetminus (\Phi_{\nu} \cap \Phi_{\s})^{+}$ and where each $x_a \in F$.  Using the fact that $u_\beta(1)$ normalizes the abelian group $N_1^-$, we obtain 
$$[u_{\beta}(1), n]= \prod_{a = 1}^m[u_{\beta}(1), u_{-\gamma_a}(x_a)]\in (N_1^-\cap K).$$

Now set $A_a : = \Phi \cap \{j \beta - \gamma_a : \, 1 \leq j \leq 2\}$ for $1\leq a \leq m$ (so that $A_a$ is a proper subset of the $\beta$-string through $-\gamma_a$).  Since root strings in type $C_n$ have length at most 3, we see that for each fixed $a$, either $A_a \cap A_{a'} = \emptyset$ for all $a' \neq a$, or $A_{a} \cap A_{a'} \neq \emptyset$ and $A_{a} \cap A_{a''} = A_{a'} \cap A_{a''} = \emptyset$ for all $a'' \neq a, a'$. If $m = 1$, or if $m = 2$ and $A_1 \cap A_2 = \emptyset$, then part (1) easily implies that $x_a \in \cO$ whenever $A_a \neq \emptyset$. If $m = 2$ and $A_1 \cap A_2 \neq \emptyset$, then neither of $\gamma_1, \gamma_2$ is equal to $\alpha_n$ and (after possibly exchanging $\gamma_1$ and $\gamma_2$) we have
 $$[u_{\beta}(1), n] =
u_{\beta - \gamma_1}(c_{\beta, -\gamma_1; 1, 1}x_1 + c_{\beta, -\gamma_2; 2, 1}x_2)\cdot u_{\beta- \gamma_2}(c_{\beta,-\gamma_2; 1, 1}x_{2}),$$
so again part (1) implies that $x_1, x_2 \in \cO$. By induction on $m$, it follows that $x_a \in \cO$ whenever $A_a \neq \emptyset$. By part (1), for each $\gamma_a \neq \alpha_n$ there exists a choice of $\beta \in (\Phi_{\nu} \cap \Phi_{\s})^{+}$ such that $A_a \neq \emptyset$.  Hence $(N_1^{-} \cap K)n = (N_1^{-} \cap K)u_{-\alpha_n}(x)$ for some $x \in F$ if for all $\beta \in  (\Phi_{\nu} \cap \Phi_{\s})^{+}$ we have $[u_{\beta}(1), n] \in (N_1^{-} \cap K)$. 
\end{proof}

We now derive a condition for the existence of a fixed point under the action of $J^*$. 

\begin{sublemma}
\label{sl4}
Continue to assume that $\alpha_n \in \Pi_{\nu}$. Let $\mu \in X_{*}(T)_{M_{\nu}, -}$, let $\zeta \in \mu_2$, and suppose that $\tS''_{\mu, 2\lambda, \zeta}$ has a fixed point under the conjugation action of $J^*$. Then $\mu = 2\lambda$ or $\mu = 2\lambda + \alpha_n^{\vee}$. 
\end{sublemma}

\begin{proof}[Proof of Sublemma \ref{sl4}]
Suppose first that $n \in N_1^{-}$ and $(N_1^{-} \cap K)n$ is fixed under conjugation by $J$. We claim that $(N_1^{-} \cap K)n = (N_1^{-} \cap K)u_{-\alpha_n}(x)$ for some $x \in \varpi^{-1}\cO$. The proof of Sublemma \ref{sl3} shows that $J$ normalizes $(N_1^{-} \cap K)$, so the commutator $[j, n]$ belongs to $(N_1^{-} \cap K)$ for each $j \in J$.  Sublemma \ref{sl3.5}(2) implies that $(N_1^{-} \cap K)n = (N_1^{-} \cap K)u_{-\alpha_n}(x)$ for some $x \in F$. Setting $j = (-\alpha_{n}^{\vee})(1 + \varpi) \in J$, we see that 
$$[j, u_{-\alpha_n}(x)] = u_{-\alpha_n}\left(x((1 + \varpi)^{2} - 1)\right)\in(N_1^{-} \cap K).$$
The difference $((1 + \varpi)^2 - 1)$ lies in $\varpi \cO \smallsetminus \varpi^{2}\cO$ since $\text{char}(\mathfrak{k}) \neq 2$, so $x \in \varpi^{-1}\cO$.

Now let $\mu \in X_{*}(T)_{M_{\nu}, -}$ and suppose that there exists a fixed point for the action of $J^*$ on $\tS''_{\mu, 2\lambda, \zeta}$.  Since the covering splits uniquely over each unipotent subgroup, the above paragraph shows that the invariant coset in $\tS''_{\mu,2\lambda,\zeta}$ is represented by $\widetilde{u}_{-\alpha_n}(x)$ for some $x \in \varpi^{-1}\cO$. By definition of $\tS''_{\mu,2\lambda,\zeta}$, we have
$$(N_1^{-} \cap K)u_{-\alpha_n}(x)\mu(\varpi) \subset (M_\nu\cap K) \lambda(\varpi)^2(M_\nu\cap K)$$ 
for some $x \in \varpi^{-1}\cO$. If $x \in \cO$, the Cartan decomposition implies that $\mu = 2\lambda$.  We may therefore assume $x \in \varpi^{-1}\cO \smallsetminus \cO$.  In this case, $u_{\alpha_n}(x^{-1})$ and $\mu(\varpi)^{-1}u_{\alpha_n}(x^{-1})\mu(\varpi) = u_{\alpha_n}(\varpi^{-\langle\alpha_n,\mu\rangle}x^{-1})$ both lie in $M_\nu\cap K$. 
This gives
\begin{eqnarray*}
(N_1^- \cap K)u_{-\alpha_n}(x)\mu(\varpi) & = & (N_1^{-} \cap K)u_{\alpha_n}(x^{-1})w_{\alpha_n}\alpha_n^{\vee}(x)u_{\alpha_n}(x^{-1})\mu(\varpi)\\
 & \subset & (M_\nu\cap K)\alpha_n^{\vee}(\varpi)^{-1}\mu(\varpi)(M_\nu\cap K),
\end{eqnarray*}
where $w_{\alpha_n} := u_{\alpha_n}(1)u_{-\alpha_n}(-1)u_{\alpha_n}(1)$ (cf. \cite{springer:linalggrps}, Lemma 8.1.4).  Therefore $(M_\nu\cap K)\alpha_n^{\vee}(\varpi)^{-1}\mu(\varpi)(M_\nu\cap K) = (M_\nu\cap K)\lambda(\varpi)^2(M_\nu\cap K)$, which implies $\mu - \alpha_n^\vee$ is in the Weyl group orbit of $2\lambda$, and in particular is contained in $2X_*(T)$.  Given $\alpha_i\in \Pi_\nu$, we have $\langle\alpha_i,\mu - \alpha_n^\vee\rangle \leq - \langle\alpha_i, \alpha_n^\vee\rangle$, which is less than or equal to 0, except if $i = n - 1$.  In this case, however, we have $\langle\alpha_{n - 1},\mu - \alpha_n^\vee\rangle \leq 1$ and $\langle\alpha_{n - 1},\mu - \alpha_n^\vee\rangle\in 2\bbZ$, which shows $\langle\alpha_{n - 1},\mu - \alpha_n^\vee\rangle \leq 0$ and consequently $\mu - \alpha_n^\vee \in X_{*}(T)_{M_{\nu}, -}$.  The Cartan decomposition now gives $\mu - \alpha_n^\vee = 2\lambda$.  
\end{proof}

Proposition \ref{satakeprop}(5), Proposition 5.1 of \cite{herzig:modpgln}, and Proposition \ref{samecoeffs} show that in order to compute $\cS^{\tM_{\nu} \cap \tM_{\s}}_{\tM_{\nu}}(T^{\tM_{\nu}}_{2\lambda})$, it suffices to know its values on $X_*(T)_{M_\nu,-}$. We have noted that the value of $\cS^{\tM_{\nu} \cap \tM_{\s}}_{\tM_{\nu}}(T^{\tM_{\nu}}_{2\lambda})$ at $\widetilde{\mu}(\varpi)$ is equal to $|\tS''_{\mu, 2\lambda, 1}| - |\tS''_{\mu, 2\lambda, -1}|$, and the above sublemmas show that this difference is nonzero for a cocharacter $\mu \in X_*(T)_{M_{\nu}, -}$ only if $\mu = 2\lambda$ or $\mu = 2\lambda + \alpha_n^{\vee}$.  Hence \eqref{redsateqn} now implies 
$$\cS^{\tM_{\nu} \cap \tM_{\s}}_{\tM_{\nu}}(T^{\tM_{\nu}}_{2\lambda}) = \begin{cases}
T^{\tM_{\nu} \cap \tM_{\s}}_{2\lambda} & \textnormal{if}~1 \leq i \leq n-1,\\
T^{\tM_{\nu} \cap \tM_\s}_{2\lambda} + d\cdot T^{\tM_{\nu} \cap \tM_\s}_{2\lambda + \alpha_n^{\vee}} & \textnormal{if}~i = n,
\end{cases}$$
where $d := |\tS''_{2\lambda + \alpha_n^{\vee}, 2\lambda, 1}| - |\tS''_{2\lambda + \alpha_n^{\vee}, 2\lambda, -1}|$. (Note that $2\lambda + \alpha_n^{\vee} \notin X_*(T)_{M_{\nu}, -}$ when $\alpha_n \in \Pi_\nu$ and $i \neq n$, and $|\tS''_{2\lambda, 2\lambda, 1}| = 1$, $|\tS''_{2\lambda, 2\lambda, -1}| = 0$.)  Therefore, we obtain
$$\cS^{\tT}_{\tG}(T_{2\lambda}) = \begin{cases}
\cS_{\tM_{\nu}\cap \tM_\s}^{\tT}(T^{\tM_{\nu} \cap \tM_{\s}}_{2\lambda}) & \textnormal{if}~1 \leq i \leq n-1,\\
\cS_{\tM_{\nu} \cap \tM_{\s}}^{\tT}(T^{\tM_{\nu} \cap \tM_{\s}}_{2\lambda} + d\cdot T^{\tM_{\nu} \cap \tM_{\s}}_{2\lambda + \alpha_n^{\vee}}) & \textnormal{if}~i = n.
\end{cases}$$

Assume first that $1\leq i \leq n - 1$.  Using Proposition 6.7 of \cite{herzig:modpgln} and Proposition \ref{samecoeffs}, we obtain
$$\cS_{\tG}^{\tT}(T_{2\lambda}) = \cS_{\tM_{\nu} \cap \tM_{\s}}^{\tT}(T^{\tM_{\nu} \cap \tM_{\s}}_{2\lambda}) = \tau_{2\lambda} - \tau_{2\lambda + \alpha_i^\vee},$$
which completes the proof in the case $1 \leq i \leq n-1$.

Finally, assume $i = n$.  In this case, $\widetilde{\lambda}(\varpi)^2$ is central in $\tM_\nu\cap \tM_\s$, and Proposition \ref{satakeprop}(5) implies
$$\cS_{\tG}^{\tT}(T_{2\lambda}) = \tau_{2\lambda} + d \cdot \cS_{\tM_\nu\cap \tM_\s}^{\tT}(T^{\tM_{\nu} \cap \tM_{\s}}_{2\lambda + \alpha_n^{\vee}}).$$
Since $\cS_{\tM_{\nu} \cap \tM_{\s}}(T^{\tM_{\nu} \cap \tM_{\s}}_{2\lambda + \alpha_n^{\vee}}) = \tau_{2\lambda + \alpha_n^{\vee}} + \sum_{\mu \neq 2\lambda + \alpha_n^{\vee}} c_{2\lambda + \alpha_{n}^{\vee}}(\mu)\tau_{\mu}$, we can use Corollary \ref{cor:oddsum} to conclude that $d = 0$. Thus, we obtain
$$\cS_{\tG}^{\tT}(T_{2\lambda}) = \tau_{2\lambda}.$$
\end{proof}

\subsection{Alternate proof of Theorem \ref{thm:changeofwt} for short simple roots}
We wish to present a more conceptual proof of the change-of-weight theorem, suggested by Florian Herzig, in the case of a fixed short simple root $\alpha_i$, $1\leq i \leq n - 1$.  As above, it suffices to prove Proposition \ref{propn:st2l}.

As in the previous proof, we take
$$\lambda := -\sum_{j = 1}^i \lambda_j,$$
so that $\langle\alpha_j,-\lambda\rangle = \delta_{i,j}$.  By Proposition \ref{satakeprop}(5), we have
\begin{equation}\label{eqsatake}
\cS_{\tG}^{\tT}(T_{2\lambda}) =  \sum_{\sub{\mu\in X_*(T)_-}{\mu \geq 2\lambda}} c_{2\lambda}(\mu)\tau_{\mu},
\end{equation}
with $c_{2\lambda}(2\lambda) = 1$.  The condition $\mu\geq 2\lambda$ gives $\mu = 2\lambda + \sum_{k = 1}^na_k\alpha_k^\vee$ for $a_k\in \bbZ_{\geq 0}$, while the condition $\mu\in X_*(T)_-$ is equivalent to
$$\sum_{k = 1}^na_k\langle\alpha_j,\alpha_k^\vee\rangle\leq 2\langle\alpha_j,-\lambda\rangle = 2\delta_{i,j}$$
for each $1 \leq j \leq n$.

We define 
 $$\cA := \left\{\vec{a} = (a_1,\ldots, a_n)\in (\bbZ_{\geq 0})^n:\sum_{k = 1}^n a_k\langle\alpha_j,\alpha_k^\vee\rangle\leq 2\langle\alpha_j, -\lambda \rangle\right\},$$
 so that equation \eqref{eqsatake} takes the form
 \begin{equation}\label{eqsatake2}
 \cS_{\tG}^{\tT}(T_{2\lambda}) =  \sum_{\vec{a}\in \cA} c_{2\lambda}(2\lambda + \vec{a}\cdot\vec{\alpha})\tau_{2\lambda + \vec{a}\cdot\vec{\alpha}}
 \end{equation}
 where $\vec{\alpha} = (\alpha_1^\vee,\ldots,\alpha_n^\vee)$.  To proceed further, we require a useful combinatorial property of the set $\cA$.  Fix $\vec{a}\in \cA$, and define 
$$\cA_{\vec{a}}^{\neq i}:= \{\vec{b} = (b_1,\ldots, b_n)\in \cA: b_j = a_j~\textnormal{for all}~j\neq i\}.$$
Since $\vec{a} \in \cA_{\vec{a}}^{\neq i}$, we have $\cA_{\vec{a}}^{\neq i}\neq \emptyset$.

\begin{lemma}
Fix $\vec{a}\in \cA$, and let $\{\vec{\varepsilon}_j: 1\leq j \leq n\}$ denote the standard basis of $\bbZ^n$.  
\begin{enumerate}
\item If $a_j = 0$ for all $j\neq i$, then $\cA_{\vec{a}}^{\neq i} = \{\vec{0},~\vec{\varepsilon}_i\}$. 
\item If $a_j\neq 0$ for some $j\neq i$, then $\cA_{\vec{a}}^{\neq i} = \{\vec{a}\}$.  
\end{enumerate}
\end{lemma}

\begin{proof}
Fix $\vec{a}\in \cA$, and let $\vec{b} = (b_1,\ldots,b_n)\in \cA$ denote the ``minimal'' element of $\cA_{\vec{a}}^{\neq i}$; that is, we have $b_j = a_j$ for all $j\neq i$, and $b_i$ is as small as possible.  Then all elements of $\cA_{\vec{a}}^{\neq i}$ are of the form $\vec{b} + m\vec{\varepsilon}_i$ with $m\geq 0$.

Let $\cC = (\langle\alpha_j,\alpha_k^\vee\rangle)_{j,k}$ denote the Cartan matrix of $\mathbf{Sp}_{2n}$, and define $\vec{c} := \cC \vec{b}$.  As $\vec{b}\in \cA$, we have $c_j\leq 2\delta_{i,j}$ for all $j$.  Assume first that $c_i \leq 0$, so that all entries of $\vec{c}$ are nonpositive.  Since all entries of $\cC^{-1}$ are nonnegative (cf. \cite{humphreys:liealg}, $\S$13, Table 1, or \cite{lusztigtits:cartan}), we see that all entries of $\vec{b} = \cC^{-1}\vec{c}$ are nonpositive.  On the other hand, $\vec{b}\in \cA$, which implies $\vec{b} = \vec{0}$.  We have $m\langle \alpha_i, \alpha_i^{\vee}\rangle \leq 2$ if and only if $m \leq 1$, so this gives $\cA_{\vec{a}}^{\neq i} = \{\vec{0}, \vec{\varepsilon}_i\}$.  

We may therefore assume $1\leq c_i\leq 2$.  The $i^{\textnormal{th}}$ entry of $\cC(\vec{b} + m\vec{\varepsilon}_i)$ is $c_i + \langle\alpha_i,\alpha_i^\vee\rangle m = c_i + 2m$.  Assuming $\vec{b} + m\vec{\varepsilon}_i\in \cA$, we have $c_i + 2m\leq 2$, which implies $m = 0$.  Therefore, $\cA_{\vec{a}}^{\neq i} = \{\vec{b}\} = \{\vec{a}\}$.
\end{proof}

We now analyze Hecke eigenvalues of principal series representations.  Let $P_i = M_i\ltimes N_i$ denote the standard parabolic subgroup corresponding to $\{\alpha_i\}$ with standard Levi subgroup $M_i$, so that $M_i\cong (F^\times)^{i - 1}\times\textnormal{GL}_2(F)\times (F^\times)^{n - i - 1}.$  Let $\xi:T\longrightarrow \fpb^\times$ denote a smooth character of $T$.  We assume throughout that $\xi\circ\alpha_i^\vee:F^\times \longrightarrow \fpb^\times$ is the trivial character, so that $\xi$ extends to a one-dimensional representation of $M_i$.  We continue to denote this representation of $M_i$ by $\xi$.  By Theorem 30 of \cite{barthellivne:irredmodp}, we have a short exact sequence of $M_i$-representations
\begin{equation}\label{levises}
0\longrightarrow \xi\longrightarrow \Ind_{B^-\cap M_i}^{M_i}(\xi)\longrightarrow \xi\otimes\textnormal{St}_{\textnormal{GL}_2}\longrightarrow 0,
\end{equation}
where $\textnormal{St}_{\textnormal{GL}_2}$ denotes the Steinberg representation of the $\textnormal{GL}_2$ factor.

We inflate the short exact sequence \eqref{levises} to $\tM_i$ to obtain
$$0\longrightarrow \xi\longrightarrow \Ind_{\tB^-\cap \tM_i}^{\tM_i}(\xi)\longrightarrow\xi\otimes\textnormal{St}_{\textnormal{GL}_2}\longrightarrow 0.$$
Here, the character $\xi$ is viewed as a character of $\tM_i$ via $\pr$, and the $\tM_i$-equivariant isomorphism $\Ind_{B^-\cap M_i}^{M_i}(\xi)\cong \Ind_{\tB^-\cap \tM_i}^{\tM_i}(\xi)$ is defined by sending $f$ to $f\circ\pr$.  We further twist the above sequence by $\chps|_{\tM_i}$ to obtain an exact sequence of genuine $\tM_i$-representations:
$$0\longrightarrow \xi\otimes\chps\longrightarrow \Ind_{\tB^-\cap \tM_i}^{\tM_i}(\xi)\otimes\chps\cong \Ind_{\tB^-\cap\tM_i}^{\tM_i}(\xi\otimes\chps)\longrightarrow\xi\otimes\textnormal{St}_{\textnormal{GL}_2}\otimes\chps\longrightarrow 0.$$
Finally, we parabolically induce this sequence to $\tG$ to get
\begin{equation}\label{redps}
0\longrightarrow \Ind_{\tP_i^-}^{\tG}(\xi\otimes\chps)\longrightarrow \Ind_{\tB^-}^{\tG}(\xi\otimes\chps)\longrightarrow\Ind_{\tP_i^-}^{\tG}(\xi\otimes\textnormal{St}_{\textnormal{GL}_2}\otimes\chps)\longrightarrow 0.
\end{equation}

Consider again the weight $\tV$ as in \S \ref{subsec:chwt1}, and assume we have a $\tK$-equivariant injection $f:\tV\longhookrightarrow \Ind_{\tP_i^-}^{\tG}(\xi\otimes\chps)|_{\tK}$ (this happens if $\tV_{(U^-\cap K)^*} \cong \xi\otimes\chps|_{\tT\cap \tK}$ as $(\tT\cap\tK)$-representations).  Composing this with the injection in the short exact sequence \eqref{redps}, we obtain
$$0\neq \Hom_{\tK}\left(\tV,\Ind_{\tB^-}^{\tG}(\xi\otimes\chps)|_{\tK}\right) \cong \Hom_{\tT\cap \tK}\left(\tV_{(U^-\cap K)^*},\xi\otimes\chps|_{\tT\cap \tK}\right).$$
We let $f_{\tT}\in\Hom_{\tT\cap \tK}(\tV_{(U^-\cap K)^*},\xi\otimes\chps|_{\tT\cap \tK})$ denote the element corresponding to $f$, characterized by $f_{\tT}(\overline{v}) = f(v)(1)$ for $v\in \tV$.  Additionally, we let 
$$\jmath:\Hom_{\tT\cap \tK}\left(\tV_{(U^-\cap K)^*},\xi\otimes\chps|_{\tT\cap \tK}\right) \stackrel{\sim}{\longrightarrow} \Hom_{\tK}\left(\tV,\Ind_{\tB^-}^{\tG}(\xi\otimes\chps)|_{\tK}\right)$$ 
denote the inverse of the above map; then $\jmath$ sends $f_{\tT}$ to $f$.  The Hecke algebra $\cH_{\tG}(\tV)$ naturally acts on the multiplicity space $\Hom_{\tK}(\tV,\Ind_{\tB^-}^{\tG}(\xi\otimes\chps)|_{\tK})$ on the right, and likewise $\cH_{\tT}(\tV_{(U^-\cap K)^*})$ acts on $\Hom_{\tT\cap \tK}(\tV_{(U^-\cap K)^*},\xi\otimes\chps|_{\tT\cap \tK})$, so the functorial properties of the Satake transform (\cite{henniartvigneras:cptparabolic}, diagram (4)) give
$$f*T_{2\lambda} = \jmath\left(f_{\tT}*\cS_{\tG}^{\tT}(T_{2\lambda})\right),$$
where $(f*T_{\lambda'})(v) = \sum_{g\in \tK\backslash\tG}g^{-1}\cdot f(T_{\lambda'}(g)(v))$ (and we have an analogous expression for $f_{\tT}*\tau_{\mu'}$; see \cite{herzig:modpgln}, $\S$2.1.2).  

Fix $v\in \tV$.  By the explicit action of $\cH_{\tT}(\tV_{(U^-\cap K)^*})$ on $f_{\tT}$, we get
\begin{eqnarray*}
(f_{\tT}*\tau_\mu)(\overline{v}) & = & \sum_{t\in (\tT\cap \tK)\backslash\tT}t^{-1}\cdot f_{\tT}(\tau_{\mu}(t)(\overline{v}))\\
 & = & \widetilde{\mu}(\varpi)^{-1}\cdot f_{\tT}(\overline{v})\\
 & = & (\xi\otimes\chps)(\widetilde{\mu}(\varpi))^{-1} f_{\tT}(\overline{v}).
\end{eqnarray*}
Therefore, using equation \eqref{eqsatake2} gives
\begin{eqnarray}
f*T_{2\lambda} & = & \jmath\left(f_{\tT}*\cS_{\tG}^{\tT}(T_{2\lambda})\right)\notag\\
 & = & \jmath\left(f_{\tT}*\left(\sum_{\vec{a}\in \cA} c_{2\lambda}(2\lambda + \vec{a}\cdot\vec{\alpha})\tau_{2\lambda + \vec{a}\cdot\vec{\alpha}}\right)\right)\notag\\
 & = & \jmath\left(\left(\sum_{\vec{a}\in\cA}c_{2\lambda}(2\lambda + \vec{a}\cdot\vec{\alpha})(\xi\otimes\chps)(\widetilde{(2\lambda + \vec{a}\cdot\vec{\alpha})}(\varpi))^{-1}\right)f_{\tT}\right)\notag\\
 & = & (\xi\otimes\chps)(\widetilde{\lambda}(\varpi))^{-2}\left(\sum_{\vec{a}\in\cA}c_{2\lambda}(2\lambda + \vec{a}\cdot\vec{\alpha})\left(\prod_{j = 1}^n(\xi\otimes\chps)(\widetilde{\alpha}_j^\vee(\varpi))^{-a_j}\right)\right)f\notag\\
 & = & (\xi\otimes\chps)(\widetilde{\lambda}(\varpi))^{-2}\left(\sum_{\textnormal{distinct}~\cA_{\vec{a}}^{\neq i}}\left(\sum_{\vec{b}\in \cA_{\vec{a}}^{\neq i}}c_{2\lambda}(2\lambda + \vec{b}\cdot\vec{\alpha})\right)\prod_{\sub{j = 1}{j\neq i}}^n(\xi\otimes\chps)(\widetilde{\alpha}_j^\vee(\varpi))^{-a_j}\right)f\label{const}
\end{eqnarray}
The fifth equality follows from the fact that $\xi\otimes\chps$ is trivial on $\widetilde{\alpha}_i^{\vee}(x)$.  Moreover, for distinct sets $\cA_{\vec{a}}^{\neq i}$, $\cA_{\vec{a}'}^{\neq i}$, the $(n - 1)$-tuples $(a_1,\ldots,a_{i - 1},a_{i + 1}, \ldots, a_n), (a_1',\ldots,a_{i - 1}',a_{i + 1}', \ldots, a_n')$ are distinct (this follows directly from the definition).

Now, the map $f:\tV\longhookrightarrow \Ind_{\tP_i^-}^{\tG}(\xi\otimes\chps)|_{\tK}$ induces a map $\boldsymbol{f}:\ind_{\tK}^{\tG}(\tV)\longrightarrow \Ind_{\tP_i^-}^{\tG}(\xi\otimes\chps)$ by Frobenius reciprocity.  This descends to a nonzero map (which we still denote $\boldsymbol{f}$)
$$\boldsymbol{f}:\chi\otimes_{\cH_{\tG}(\tV)}\ind_{\tK}^{\tG}(\tV) \longrightarrow \Ind_{\tP_i^-}^{\tG}(\xi\otimes\chps),$$
where $\chi$ denotes the character obtained by composing the Satake transform $\cS_{\tG}^{\tT}:\cH_{\tG}(\tV)\longrightarrow\cH_{\tT}(\tV_{(U^-\cap K)^*})$ with the character $\tau_\mu\longmapsto (\xi\otimes\chps)(\widetilde{\mu}(\varpi))^{-1}$.  In addition, the maps $\varphi^-,\varphi^+$ induce maps $\boldsymbol{\varphi}^-, \boldsymbol{\varphi}^+$ as follows:

\centerline{
\xymatrix{ \chi\otimes_{\cH_{\tG}(\tV)}\ind_{\tK}^{\tG}(\tV) \ar@<1ex>[r]^{\boldsymbol{\varphi}^+} & \chi\otimes_{\cH_{\tG}(\tV')}\ind_{\tK}^{\tG}(\tV') \ar@<1ex>[l]^{\boldsymbol{\varphi}^-}
}
}
\noindent whose composite is the constant \eqref{const}. (Here we identify $\cH_{\tT}(\tV_{(U^-\cap K)^*})$ and $\cH_{\tT}(\tV'_{(U^-\cap K)^*})$.)

We claim that the constant \eqref{const} is equal to 0.  Suppose not.  Then $\boldsymbol{\varphi}^+$ and $\boldsymbol{\varphi}^-$ are isomorphisms, and we obtain a nonzero map 
$$\boldsymbol{f}\circ\boldsymbol{\varphi}^-:\chi \otimes_{\cH_{\tG}(\tV')}\ind_{\tK}^{\tG}(\tV')\longrightarrow \Ind_{\tP_i^-}^{\tG}(\xi\otimes\chps),$$
which gives, by Frobenius reciprocity, an injection $\tV'\longhookrightarrow \Ind_{\tP_i^-}^{\tG}(\xi\otimes\chps)|_{\tK}$.  However, such an injection implies that $\Hom_{\tM_i\cap \tK}(\tV'_{(N_i^-\cap K)^*},\xi\otimes\chps|_{\tM_i\cap \tK})\neq 0$, which shows that $\tV'_{(N_i^-\cap K)^*}$ is one-dimensional.  This contradicts the construction of $\tV'$.

Therefore, we have that the constant \eqref{const} equals 0, for all $\xi$ such that $\xi\circ\alpha_i^\vee$ is the trivial character and $\tV_{(U^-\cap K)^*}\cong \xi\otimes\chps|_{\tT\cap\tK}$ as $(\tT\cap\tK)$-representations.  Viewing the constant \eqref{const} as a polynomial in the variables $(\xi\otimes\chps)(\widetilde{\alpha}_j^\vee(\varpi))$ for $j\neq i$, we see that this polynomial vanishes at every point of $(\fpb^\times)^{n - 1}$, and therefore must be identically zero by the Nullstellensatz.  Using the description of the sets $\cA_{\vec{a}}^{\neq i}$, this gives
\begin{eqnarray*}
0 & = & \sum_{\vec{b}\in \cA_{\vec{a}}^{\neq i}}c_{2\lambda}(2\lambda + \vec{b}\cdot\vec{\alpha})\\
 & = & \begin{cases}c_{2\lambda}(2\lambda) + c_{2\lambda}(2\lambda + \alpha_i^\vee) = 1 + c_{2\lambda}(2\lambda + \alpha_i^\vee)& \textnormal{if}~a_j = 0~\textnormal{for all}~j\neq i,\\ c_{2\lambda}(2\lambda + \vec{a}\cdot\vec{\alpha})& \textnormal{if}~a_j \neq 0~\textnormal{for some}~j\neq i.\end{cases}
\end{eqnarray*}
We conclude 
$$\cS_{\tG}^{\tT}(T_{2\lambda}) = \tau_{2\lambda} - \tau_{2\lambda + \alpha_i^\vee}.$$

\section{Classification}\label{classnchapter}

We now begin our classification of irreducible admissible genuine representations of $\tG$.  We follow the reductive case closely (cf.  \cite{abehenniartherzigvigneras:irredmodp}, \cite{abe:irredmodp}, \cite{herzig:modpgln}).

\subsection{Setup}

\begin{defn}
Given a simple root $\alpha\in \Pi$, we define $M'_\alpha\subset\tG$ to be the subgroup generated by $U_\alpha^*$ and $U_{-\alpha}^*$. 
\end{defn}

\begin{lemma}\label{malpha}
We have
$$M'_\alpha\cong\begin{cases}\textnormal{SL}_2(F) & \textnormal{if}~\alpha\neq\alpha_n,\\ \widetilde{\textnormal{SL}}_2(F) & \textnormal{if}~\alpha = \alpha_n. \end{cases}$$
\end{lemma}

\begin{proof}
Let $\alpha\in\Pi$, and let $\varphi_\alpha:\textnormal{SL}_2(F)\longrightarrow G$ denote the morphism associated to $\alpha$.  Pulling back the cover $\tG\longrightarrow G$ to $\textnormal{SL}_2(F)$ via $\varphi_\alpha$ gives the diagram:

\centerline{
\xymatrix{
1\ar[r] & \mu_2\ar[r]&  \tG \ar^{\pr}[r]\ar@{}[dr]|{\square} & G\ar[r] & 1\\
1\ar[r] & \mu_2\ar[r]\ar@{=}[u] & \tH \ar[r]\ar[u] & \textnormal{SL}_2(F)\ar[r]\ar[u]_{\varphi_\alpha} & 1
}
} 
\noindent The cover $\tH\longrightarrow \textnormal{SL}_2(F)$ may be realized by an extension of algebraic groups (as in $\S$\ref{cover}), and is therefore determined by the restriction of the quadratic form $Q$ to $\bbZ\alpha^\vee$.  Since $Q(\alpha_i^\vee) = 2$ for $1\leq i \leq n - 1$ and $Q(\alpha_n^\vee) = 1$, we see that the extension 
$$1\longrightarrow \mu_2\longrightarrow \tH \longrightarrow \textnormal{SL}_2(F) \longrightarrow 1$$
defined by $\varphi_\alpha$ splits if and only if $\alpha \neq \alpha_n$.  This gives the claim. 

Alternatively, one may obtain the result by computing the extension class of the pullback using Steinberg cocycles (cf. \cite{moore:gpextensions}).  
\end{proof}

The following lemma will be useful in the definition of a supersingular triple at the end of this section.

\begin{lemma}\label{extlemma}
Let $\tM$ be a Levi subgroup of $\tG$, and let $\Pi_M = \Pi_1\sqcup\Pi_2$ be a partition of $\Pi_M$ such that $\langle\Pi_1,\Pi_2^\vee\rangle = 0$.  Let $\tM_i$ denote the Levi subgroup corresponding to $\Pi_i$, and let $L'_2$ denote the subgroup of $\tT\subset\tM_1$ generated by $\tT\cap M'_\beta$ for $\beta\in \Pi_2$.  We then have
$$\tM/[\tM_2,\tM_2]\cong \tM_1/L'_2.$$
\end{lemma}

\begin{proof}
This follows easily from the proof of Lemma 3.2 of \cite{abe:irredmodp}.  Note that we must split up the proof into the cases $\alpha_n\in\Pi_2$ and $\alpha_n\not\in\Pi_2$.  
\end{proof}

Recall the set $\Pi(\chi)$ used in the statement of the change-of-weight theorem (equation \eqref{defofpichi}).  Given a standard Levi subgroup $\tM$, a weight $\tV$ of $\tM$, and a character $\chi:\cH_{\tM}(\tV)\longrightarrow \fpb$, we let $\Pi_M(\chi)$ denote the analogously defined subset of $\Pi_M$.

\begin{defn}
\label{supersingular}
Let $\tM$ be a standard Levi subgroup of $\tG$ and $\sigma$ a representation of $\tM$.  We say $\sigma$ is \textit{supersingular (with respect to $(\tK\cap \tM,\tT,\tB\cap\tM)$)} if it is irreducible, admissible and genuine, and the following condition is satisfied: for all weights $\tV$ of $\tM$ and for all characters $\chi:\cH_{\tM}(\tV)\longrightarrow \fpb$ such that there exists a nonzero map $\chi\otimes_{\cH_{\tM}(\tV)}\ind_{\tM\cap\tK}^{\tM}(\tV)\longrightarrow\sigma$, we have $\Pi_M(\chi) = \Pi_M$.     
\end{defn}

\begin{remark}
One can similarly define supersingularity with respect to $(\tK' \cap \tM, \tT, \tB \cap \tM)$, where $\tK'$ is a maximal compact subgroup of $\tG$ defined by another hyperspecial point in the apartment corresponding to $T$. We will show that the notion of supersingularity is independent of such a choice of maximal compact $\tK'$; see Lemma \ref{kindep}. In the meantime, we will often abbreviate ``supersingular with respect to $(\tK \cap \tM, \tT, \tB \cap \tM)$'' to ``supersingular.'' 
\end{remark}

As in \cite{abe:irredmodp}, we begin with a standard parabolic subgroup $\tP = \tM\ltimes N^{*}$, and a genuine irreducible admissible supersingular representation $\sigma$ of $\tM$.  We define
$$\Pi(\sigma) := \left\{\alpha\in \Pi:\begin{array}{l}\diamond~ \langle\Pi_M,\alpha^\vee\rangle = 0\\ \diamond~ \textnormal{the group}~\tT\cap M'_\alpha~\textnormal{acts trivially on}~\sigma \end{array} \right\}.$$

\begin{remark}
Note that, unlike the reductive case, we can never have $\alpha_n\in \Pi(\sigma)$ (this follows from Lemma \ref{malpha} and the fact that $\sigma$ is genuine).  
\end{remark}

\begin{remark}\label{factorcoroot}
Since $G$ is simply connected, we have $\tT\cap M'_\alpha = (\tT\cap M'_\alpha\cap \tK)\times\widetilde{\alpha}^\vee(\varpi)^{\bbZ}$.  In particular, $\tT\cap M'_\alpha$ acts trivially on $\sigma$ if and only if $\tT\cap M'_\alpha\cap \tK$ and $\widetilde{\alpha}^\vee(\varpi)$ act trivially.  
\end{remark}

Let $\tP(\sigma) = \tM(\sigma)\ltimes N(\sigma)^{*}$ denote the standard parabolic subgroup corresponding to $\Pi_M\sqcup \Pi(\sigma)$, and let $\tM_1$ denote the Levi subgroup corresponding to the subset $\Pi(\sigma)$.  Since $\sigma$ is trivial on $\tT\cap M'_\alpha$ for every $\alpha\in \Pi(\sigma)$, Lemma \ref{extlemma} implies that we may extend $\sigma$ to a genuine representation of $\tM(\sigma)$ on which $[\tM_1,\tM_1]$ acts trivially.  We further inflate this representation to $\tP(\sigma)^-$ and denote it by ${}^e\sigma$; it is irreducible, admissible, and genuine.

Recall that for any pair of standard parabolic subgroups $P\subset Q$ of $G$, we have the generalized Steinberg representation of $Q^-$:
$$\St_{P^-}^{Q^-}:= \Ind_{P^-}^{Q^-}(1)\Big/\sum_{P^-\subsetneq Q'^-\subset Q^-}\Ind_{Q'^-}^{Q^-}(1),$$
where $1$ denotes the trivial representation of each respective group.  The unipotent radical of $Q^-$ acts trivially on $\St_{P^-}^{Q^-}$, and it is known that $\St_{P^-}^{Q^-}$ is irreducible and admissible (\cite{ly:steinberg}, Th\'eor\`eme 3.1).  We shall view $\St_{P^-}^{Q^-}$ as a (nongenuine) representation of $\tQ^-$ by inflation; the inflated representation continues to be irreducible and admissible.  Moreover, the argument of \cite{abe:irredmodp} Lemma 5.2 and the discussion preceding it show that $\St_{Q^-}^{P(\sigma)^-}|_{[\tM_1,\tM_1]}$ is irreducible and admissible.

We consider triples $(\tP^-,\sigma,\tQ^-)$, where $\tP^-$ and $\sigma$ are as above, and $\tQ^-$ is a parabolic subgroup of $\tG$ satisfying $\tP^-\subset \tQ^-\subset \tP(\sigma)^-$.  We call such a triple a \textit{supersingular triple}.  Two supersingular triples $(\tP^-,\sigma,\tQ^-), (\tP'^-,\sigma',\tQ'^-)$ are defined to be equivalent if $\tP'^- = \tP^-$, $\tQ'^- = \tQ^-$ and $\sigma'\cong \sigma$.  To a triple $(\tP^-,\sigma,\tQ^-)$, we associate the genuine representation ${}^e\sigma\otimes\St_{Q^-}^{P(\sigma)^-}$ (which is irreducible and admissible by \cite{abe:irredmodp}, Lemmas 3.23 and 5.3).  We then set 
$$I(\tP^-,\sigma,\tQ^-) := \Ind_{\tP(\sigma)^-}^{\tG}\left({}^e\sigma\otimes\St_{Q^-}^{P(\sigma)^-}\right).$$
The properties of parabolic induction imply that $I(\tP^-,\sigma,\tQ^-)$ is admissible and genuine.

\subsection{Main results}

We proceed in several stages.

\begin{propn}\label{pichi}
Let $(\tP^-,\sigma,\tQ^-)$ be a supersingular triple, and let $\varphi\in \Hom_{\tK}(\tV, I(\tP^-,\sigma,\tQ^-)|_{\tK})$ denote a nonzero eigenvector for $\cH_{\tG}(\tV)$, with associated eigenvalues $\chi:\cH_{\tG}(\tV)\longrightarrow \fpb$.  Then $\Pi(\chi) = \Pi_M$.  
\end{propn}

\begin{proof}
This follows in exactly the same way as III.18 Corollary of \cite{abehenniartherzigvigneras:irredmodp}.  Note that necessary properties of $\St_{P^-}^{Q^-}$ follow from the reductive case, since this representation is nongenuine, and therefore descends to the reductive quotient.  
\end{proof}

\begin{propn}\label{irred}
Given a supersingular triple $(\tP^-,\sigma,\tQ^-)$, the representation $I(\tP^-,\sigma,\tQ^-)$ is irreducible, admissible, and genuine.  
\end{propn}

\begin{proof}
It only remains to prove that $I(\tP^-,\sigma,\tQ^-)$ is irreducible.  Let $\tV$ be a weight appearing in $I(\tP^-,\sigma,\tQ^-)|_{\tK}$, and assume that $\varphi\in \textnormal{Hom}_{\tK}(\tV,I(\tP^-,\sigma,\tQ^-)|_{\tK})$ is a nonzero eigenvector for $\cH_{\tG}(\tV)$, the existence of which follows from admissibility.  Let $\chi:\cH_{\tG}(\tV)\longrightarrow \fpb$ denote the set of eigenvalues of $\varphi$.  The map $\varphi$ induces a nonzero $\tG$-morphism
$$\boldsymbol{\varphi}:\ind_{\tK}^{\tG}(\tV)\longrightarrow I(\tP^-,\sigma,\tQ^-),$$
and we let
$$\boldsymbol{\varphi}_{\tM(\sigma)}:\ind_{\tM(\sigma)\cap \tK}^{\tM(\sigma)}(\tV_{(N(\sigma)^-\cap K)^*})\longtwoheadrightarrow {}^e\sigma\otimes\St_{Q^-}^{P(\sigma)^-}$$
denote the $\tM(\sigma)$-equivariant map associated to $\boldsymbol{\varphi}$ (cf. \cite{henniartvigneras:cptparabolic}, equation (2)), which is surjective by irreducibility of ${}^e\sigma\otimes\St_{Q^-}^{P(\sigma)^-}$.  Since parabolic induction is exact, we obtain a $\tG$-equivariant surjection
$$\Ind_{\tP(\sigma)^-}^{\tG}(\boldsymbol{\varphi}_{\tM(\sigma)}):\Ind_{\tP(\sigma)^-}^{\tG}\left(\ind_{\tM(\sigma)\cap \tK}^{\tM(\sigma)}(\tV_{(N(\sigma)^-\cap K)^*})\right)\longtwoheadrightarrow I(\tP^-,\sigma,\tQ^-).$$

We claim that $\varphi(\tV)$ generates $I(\tP^-,\sigma,\tQ^-)$.  We write $\tV = \tF(\nu)$ for $\nu\in X_q(T)$, and use induction on $|\Pi_\nu\smallsetminus(\Pi_M\sqcup \Pi(\sigma))|$.  If this set has size 0, we have $\Pi_\nu\subset \Pi_M\sqcup \Pi(\sigma)$, and Proposition \ref{compisom} and the discussion above show that we have a surjection
$$\cH_{\tM(\sigma)}(\tV_{(N(\sigma)^-\cap K)^*})\otimes_{\cH_{\tG}(\tV)}\ind_{\tK}^{\tG}(\tV)\longtwoheadrightarrow I(\tP^-,\sigma,\tQ^-).$$
By Proposition \ref{satakeprop}(3) and the diagram on p. 465 of \textit{loc. cit.}, this map is the localization of $\boldsymbol{\varphi}$ at an element whose Satake transform is central and invertible in $\cH_{\tM(\sigma)}(\tV_{(N(\sigma)^-\cap K)^*})$.  Therefore, diagram (4) of \textit{loc. cit.} shows that the character $\chi$ extends to $\cH_{\tM(\sigma)}(\tV_{(N(\sigma)^-\cap K)^*})$, and the localization of $\boldsymbol{\varphi}$ descends to a surjection (which we denote by the same letter)
$$\boldsymbol{\varphi}:\chi\otimes_{\cH_{\tG}(\tV)}\ind_{\tK}^{\tG}(\tV)\longtwoheadrightarrow I(\tP^-,\sigma,\tQ^-).$$
We conclude that $\varphi(\tV)$ generates $I(\tP^-,\sigma,\tQ^-)$.

Assume now that $|\Pi_\nu\smallsetminus(\Pi_M\sqcup \Pi(\sigma))| > 0$, and fix $\alpha\in \Pi_\nu\smallsetminus(\Pi_M\sqcup\Pi(\sigma))$.  Since $\Pi_M = \Pi(\chi)$ by Proposition \ref{pichi}, we have that $\alpha\not\in\Pi(\chi)$.  Suppose for the moment that $\alpha\neq \alpha_n$, and that $\langle\Pi(\chi),\alpha^\vee\rangle = 0$.  We then have $\alpha^\vee\in X_*(T)_{M,-}$ and $\cS_{\tM}^{\tT}(T_{\alpha^\vee}^{\tM}) = \tau_{\alpha^\vee}$ (cf. equation \eqref{satakeeq}).  By III.18 Proposition of \cite{abehenniartherzigvigneras:irredmodp}, the $\cH_{\tG}(\tV)$-eigenvector $\varphi$ descends to the $\cH_{\tM}(\tV_{(N^-\cap K)^*})$-eigenvector $\varphi_M:\tV_{(N^-\cap K)^*}\longhookrightarrow \sigma|_{\tK\cap \tM}$ (with corresponding eigenvalues $\chi'$), so that 
$$(\chi'\circ(\cS_{\tM}^{\tT})^{-1})(\tau_{\alpha^\vee})\varphi_M(\overline{v}) = \chi'(T_{\alpha^\vee}^{\tM})\varphi_M(\overline{v}) = (\varphi_M* T_{\alpha^\vee}^{\tM})(\overline{v}) = \widetilde{\alpha}^{\vee}(\varpi)^{-1}\cdot(\varphi_M(\overline{v})),$$
where $v\in \tV$, and where the last equality follows from the fact that $\widetilde{\alpha}^\vee(\varpi)$ is central in $\tM$.   Using the fact that $\alpha\neq \alpha_n$, Lemma 5.5 of \cite{abe:irredmodp} shows that $\tT\cap M'_\alpha\cap \tK$ acts trivially on $\sigma$.

Suppose now only that $\alpha\neq \alpha_n$.  Since $\alpha\not\in\Pi(\sigma)$, either $\langle\Pi(\chi),\alpha^\vee\rangle\neq 0$, or $\langle\Pi(\chi),\alpha^\vee\rangle = 0$ and the subgroup $\tT\cap M_\alpha'$ acts nontrivially on $\sigma$.  In the latter case, the above discussion and Remark \ref{factorcoroot} show that $\widetilde{\alpha}^\vee(\varpi)$ acts nontrivially, and consequently $(\chi'\circ(\cS_{\tM}^{\tT})^{-1})(\tau_{\alpha^\vee})\neq 1$.  We are therefore in a position to apply Theorem \ref{thm:changeofwt} in all cases, and the result follows by induction.  
\end{proof}

\begin{cor}\label{parabindjh}
Let $\tP = \tM\ltimes N^{*}$ denote a standard parabolic subgroup of $\tG$, and $\sigma$ an irreducible admissible genuine supersingular representation of $\tM$.   Then the composition factors of $\Ind_{\tP^-}^{\tG}(\sigma)$ are given by $\{I(\tP^-,\sigma,\tQ^-):\tP^-\subset \tQ^- \subset \tP(\sigma)^-\}$.  
\end{cor}

\begin{proof}
By definition of the parabolic subgroup $\tP(\sigma)^- = \tM(\sigma)\ltimes N(\sigma)^{-,*}$, we have an extension ${}^e\sigma$ to $\tM(\sigma)$, which satisfies $\Ind_{\tP^-\cap \tM(\sigma)}^{\tM(\sigma)}(\sigma)\cong {}^e\sigma\otimes\Ind_{\tP^-\cap \tM(\sigma)}^{\tM(\sigma)}(1)$.  Since $\Ind_{\tP^-\cap \tM(\sigma)}^{\tM(\sigma)}(1)$ is a nongenuine representation inflated from the representation $\Ind_{P^-\cap M(\sigma)}^{M(\sigma)}(1)$ of $M(\sigma)$, the composition factors of ${}^e\sigma\otimes\Ind_{\tP^-\cap\tM(\sigma)}^{\tM(\sigma)}(1)$ are given by $\{{}^e\sigma\otimes\St_{Q^-}^{P(\sigma)^-}: P^-\subset Q^-\subset P(\sigma)\}$ (\cite{herzig:modpgln}, Corollary 7.3).  Since parabolic induction is exact, the result follows.  
\end{proof}

\begin{propn}\label{inj}
Assume $(\tP^-, \sigma, \tQ^-)$ and $(\tP'^-, \sigma', \tQ'^-)$ are two supersingular triples such that 
$$I(\tP^-, \sigma, \tQ^-)\cong I(\tP'^-, \sigma', \tQ'^-).$$
Then $(\tP^-, \sigma, \tQ^-)$ and $(\tP'^-, \sigma', \tQ'^-)$ are equivalent.  
\end{propn}

\begin{proof}
This follows in the same way as the proof in $\S$III.24 of \cite{abehenniartherzigvigneras:irredmodp}, using the functor of ordinary parts.  
\end{proof}

\begin{propn}\label{surj}
Let $\pi$ be an irreducible admissible genuine representation of $\tG$.  Then there exists a supersingular triple $(\tP^-,\sigma,\tQ^-)$ such that $\pi\cong I(\tP^-,\sigma,\tQ^-)$.  
\end{propn}

\begin{proof}
Assume first that for every weight $\tV$ of $\pi$ and every set of Hecke eigenvalues $\chi$, we have $\Pi(\chi) = \Pi$.  This implies that $\pi$ is supersingular and $\pi\cong I(\tG, \pi, \tG)$.  We may therefore assume that we are given $\tV = \tF(\nu)$ and $\chi$ for which $\Pi(\chi)\neq \Pi$, so that we have a surjective morphism 
$$\chi\otimes_{\cH_{\tG}(\tV)}\ind_{\tK}^{\tG}(\tV)\longtwoheadrightarrow \pi.$$
We write $\chi = \chi'\circ\cS_{\tG}^{\tM_{\Pi(\chi)}}$ for some character $\chi':\cH_{\tM_{\Pi(\chi)}}(\tV_{(N_{\Pi(\chi)}^-\cap K)^*})\longrightarrow \fpb$.  We may further assume that $|\Pi_\nu|$ is as small as possible.

Suppose $\alpha_i\in \Pi_\nu\smallsetminus\Pi(\chi)$, and set $\nu' = \nu + (q - 1)\omega_{\alpha_i}$.  We assume that either $i = n$, or that $1 \leq i \leq n - 1$ and one of the following holds: (1) $\langle\Pi(\chi),\alpha_i^\vee\rangle \neq 0$; (2) $(\chi'\circ(\cS_{\tM_{\Pi(\chi)}}^{\tT})^{-1})(\tau_{\alpha_i^\vee}) \neq 1$.  In both cases, Theorem \ref{thm:changeofwt} shows that we may replace $\Pi_{\nu}$ with $\Pi_{\nu'}$, which contradicts the minimality of $|\Pi_\nu|$.  Therefore, we have $\alpha_n\not\in\Pi_\nu\smallsetminus\Pi(\chi)$, and for every $\alpha\in \Pi_\nu\smallsetminus\Pi(\chi)$ we have $\langle\Pi(\chi),\alpha^\vee\rangle = 0$ and $(\chi'\circ(\cS_{\tM_{\Pi(\chi)}}^{\tT})^{-1})(\tau_{\alpha^\vee}) = 1$.

Assume that $\Pi_\nu \cup \Pi(\chi) = \Pi$.  By the above, we have $\langle\Pi(\chi),(\Pi\smallsetminus\Pi(\chi))^\vee\rangle = 0$.  Since the root system of $\textbf{Sp}_{2n}$ is irreducible, this implies $\Pi(\chi) = \emptyset$ and $\Pi_\nu = \Pi$.  However, the above paragraph shows $\alpha_n\not\in\Pi_\nu$, and we have a contradiction.  We may therefore assume $\Pi_\nu\cup \Pi(\chi) \neq \Pi$.

We now induct on $n$.  For $n = 1$, we have $\Pi_\nu = \Pi(\chi) = \emptyset$, and Proposition \ref{compisom} gives
$$\Ind_{\tB^-}^{\tG}(\xi)\stackrel{\sim}{\longrightarrow}\chi\otimes_{\cH_{\tG}(\tV)}\ind_{\tK}^{\tG}(\tV)\longtwoheadrightarrow \pi,$$
where $\xi$ is the genuine character of $\tT$ given by $\chi'\otimes_{\cH_{\tT}(\tV_{(U^-\cap K)^*})}\ind_{\tT\cap\tK}^{\tT}(\tV_{(U^-\cap K)^*})$.  By Proposition \ref{irred}, $\Ind_{\tB^-}^{\tG}(\xi) = I(\tB^-,\xi,\tB^-) $ is irreducible, and the above surjection is an isomorphism.

Assume now that $n > 1$, and let $\tP = \tM\ltimes N^{*}$ denote the standard parabolic subgroup associated to $\Pi_\nu\cup\Pi(\chi)$.  Note that $\tV$ is $\tM$-regular by construction.  Again applying Proposition \ref{compisom} gives
$$\Ind_{\tP^-}^{\tG}\left(\chi'\circ\cS_{\tM}^{\tM_{\Pi(\chi)}}\otimes_{\cH_{\tM}(\tV_{(N^-\cap K)^*})}\ind_{\tM\cap\tK}^{\tM}(\tV_{(N^-\cap K)^*})\right)\stackrel{\sim}{\longrightarrow}\chi\otimes_{\cH_{\tG}(\tV)}\ind_{\tK}^{\tG}(\tV)\longtwoheadrightarrow \pi.$$
As in Lemma 9.9 of \cite{herzig:modpgln}, this implies that there exists an irreducible admissible genuine representation $\sigma$ of $\tM$ and a $\tG$-equivariant surjection
$$\Ind_{\tP^-}^{\tG}(\sigma)\longtwoheadrightarrow \pi.$$

The assumption $\Pi_\nu\cup \Pi(\chi)\neq \Pi$ implies that the Levi subgroup $M$ has the form 
$$M \cong \textnormal{GL}_{n_1}(F)\times\cdots\times\textnormal{GL}_{n_r}(F)\times\textnormal{Sp}_{2m}(F)$$
with $r, m, n_i\geq1$ and $m + \sum_{i = 1}^r n_i = n$.  The preimage $\tM$ therefore takes the form $\tM \cong \tH_1\times_{\mu_2}\tH_2$, where $\tH_1$ is the preimage of $\textnormal{GL}_{n_1}(F)\times\cdots\times\textnormal{GL}_{n_r}(F)\subset M_\s$, and $\tH_2 \cong \widetilde{\textnormal{Sp}}_{2m}(F)$.  One sees easily that every irreducible admissible genuine representation of $\tM$ is of the form $\sigma_1\boxtimes \sigma_2$, where $\sigma_i$ is an irreducible admissible genuine representation of $\tH_i$.  By the discussion in $\S$\ref{twistrep} and Theorem 9.8 of \cite{herzig:modpgln}, $\sigma_1$ is a twist by $\chi_\psi$ of an irreducible admissible representation of $H_1 = \textnormal{GL}_{n_1}(F)\times\cdots\times\textnormal{GL}_{n_r}(F)$, and therefore has the desired form.  Likewise, by induction we get that $\sigma_2$ is of the desired form.  Therefore, there exists a parabolic subgroup $\tP_0^- = \tM_0\ltimes N_0^{-,*}$ and an irreducible admissible genuine supersingular representation $\sigma_0$ of $\tM_0$ such that $\pi$ is a subquotient of $\Ind_{\tP_0^-}^{\tG}(\sigma_0)$.  By Corollary \ref{parabindjh}, the result follows.
\end{proof}

We arrive at our main theorem:

\begin{thm}
\label{classn}
The map $(\tP^-,\sigma,\tQ^-)\longmapsto I(\tP^-,\sigma,\tQ^-)$ gives a bijection between equivalence classes of supersingular triples and isomorphism classes of irreducible admissible genuine representations of $\tG$.  
\end{thm}

\begin{proof}
This follows by collecting together Propositions \ref{irred}, \ref{inj}, and \ref{surj}.  
\end{proof}

\subsection{Corollaries}

We now explore several consequences of the above results.

Let $\tM$ be a standard Levi subgroup of $\tG$.  We say that a representation $\pi$ of $\tM$ is \textit{supercuspidal} if $\pi$ is irreducible, admissible, and genuine, and if $\pi$ is not isomorphic to a subquotient of a parabolic induction $\Ind_{\tQ^{-}}^{\tM}(\sigma)$ for any proper parabolic subgroup $\tQ$ of $\tM$ (not necessarily standard) and any irreducible admissible genuine representation $\sigma$ of the Levi factor of $\tQ$.

\begin{propn}\label{ssscG}
A representation $\pi$ of $\tG$ is supersingular with respect to $(\tK,\tT,\tB)$ if and only if it is supercuspidal.
\end{propn}

\begin{proof}
This follows in same way as VI.2 Theorem of \cite{abehenniartherzigvigneras:irredmodp}, using Theorem \ref{classn} and Corollary \ref{parabindjh}.
\end{proof}

\begin{cor}\label{sssc}
Let $\tM$ be a standard Levi subgroup of $\tG$. A representation $\sigma$ of $\tM$ is supersingular with respect to $(\tK \cap \tM, \tT, \tB \cap \tM)$ if and only if $\sigma$ is supercuspidal.  
\end{cor}

\begin{proof}
We have $\tM \cong \tH_1 \times_{\mu_2} \tH_2$ as in the final step of the proof of Proposition \ref{surj}, so that $H_1 \cong \textnormal{GL}_{n_1}(F)\times\cdots\times\textnormal{GL}_{n_r}(F)$ and $H_2\cong \textnormal{Sp}_{2m}(F)$ (where we now allow either one of $H_1$ or $H_2$ to be trivial). The representation $\sigma$ decomposes as
$$\sigma \cong (\sigma_1\otimes\chi_\psi)\boxtimes\sigma_2,$$
where $\sigma_1$ is an irreducible admissible representation of $H_1$ inflated to $\tH_1$, and $\sigma_2$ is an irreducible admissible genuine representation of $\tH_2$.  One easily sees that $\sigma$ is supersingular with respect to $(\tK \cap \tM, \tT, \tB \cap \tM)$ (resp. $\sigma$ is supercuspidal) if and only if $\sigma_1$ is supersingular with respect to $(K \cap H_1, T, B \cap H_1)$ and $\sigma_2$ is supersingular with respect to $(\tK \cap \tH_2, \tT, \tB \cap \tH_2)$ (resp. both $\sigma_1$ and $\sigma_2$ are supercuspidal).  The result then follows from Proposition \ref{ssscG} and VI.2 Theorem of \cite{abehenniartherzigvigneras:irredmodp}.  
\end{proof}

Next, we discuss the dependence of the notion of supersingularity on the choice of the triple $(\tK,\tT,\tB)$. Let $(\tK', \tT',\tB')$ denote any other such triple (with $\tK'$ the preimage of some hyperspecial maximal compact subgroup, $\tT'$ the preimage of a split maximal torus, etc.).  There exists some $\widetilde{\gamma} \in \widetilde{\textnormal{GSp}}_{2n}(F)$ such that 
$$(\tK',\tT',\tB') = (\widetilde{\gamma}\tK\widetilde{\gamma}^{-1}, \widetilde{\gamma}\tT\widetilde{\gamma}^{-1}, \widetilde{\gamma}\tB\widetilde{\gamma}^{-1}),$$ 
and it is clear that conjugation by $\widetilde{\gamma}$ induces a bijection $\tV \longmapsto \tV^{\widetilde{\gamma}}$ of weights as well as isomorphisms $\cH_{\tG}(\tV) \longrightarrow \cH_{\tG}(\tV^{\widetilde{\gamma}})$  (and likewise for Levi subgroups) which are compatible with Satake transforms on each side. One may then replace $(\tK,\tT,\tB)$ with $(\tK',\tT',\tB')$ in Definition \ref{supersingular} and obtain the following:

\begin{lemma}
\label{kindep}
A representation $\pi$ of $\tG$ is supersingular with respect to $(\tK, \tT, \tB)$ if and only if $\pi$ is supersingular with respect to $(\tK', \tT', \tB')$ as above.  If $\tK'$ is a conjugate of $\tK$ by an element of the maximal torus of $\widetilde{\textnormal{GSp}}_{2n}(F)$ containing $\tT$, then a representation $\sigma$ of a standard Levi subgroup $\tM$ of $\tG$ is supersingular with respect to $(\tK\cap \tM,\tT,\tB\cap \tM)$ if and only if it is supersingular with respect to $(\tK'\cap \tM,\tT,\tB\cap \tM)$.  Thus the results of Theorem \ref{classn} do not depend on the choice of such $\tK'$. 
\end{lemma}

\begin{proof}
The first part follows from applying Proposition \ref{ssscG} twice (for the triples $(\tK,\tT,\tB)$ and $(\tK',\tT',\tB')$, respectively), and noting that the notion of supercuspidality does not depend on the choice of triple $(\tK,\tT,\tB)$.  It therefore only remains to prove the second part.  We may decompose $\sigma$ as 
$$\sigma\cong(\sigma_1\otimes\chi_\psi)\boxtimes\sigma_2,$$ 
as in Corollary \ref{sssc}; the result then follows by applying the first part of the present lemma to $\sigma_2$, and applying the analogous result in the reductive case (cf. \cite{abe:irredmodp}, Corollary 5.13) to $\sigma_1$.
\end{proof}

Finally, we use Theorem \ref{classn} to describe representations of $\tG$ induced from the Siegel Levi.

\begin{lemma}
\label{msirred}
Let $\sigma$ be an irreducible admissible representation of $M_{\s}$. Then $\Ind_{\tP_{\s}^{-}}^{\tG}(\sigma\otimes\chi_\psi)$ is an irreducible admissible genuine representation of $\tG$, where $\sigma$ is viewed as a nongenuine representation of $\tM_\s$ by inflation.  
\end{lemma}

\begin{proof}
By the classification theorem for $\textnormal{GL}_n(F)$, there exists a standard parabolic subgroup $P \subset M_{\s}$ with Levi factor $M$, a supersingular representation $\rho$ of $M$, and a parabolic subgroup $Q$ with $P \subset Q \subset P(\rho)_\s$ such that 
$$\sigma \cong \Ind_{P(\rho)_\s^{-}}^{M_\s}\left({}^e\rho \otimes \St_{Q^{-}}^{P(\rho)_{\s}^{-}}\right).$$ 
Here $P(\rho)_{\s}$ is the standard parabolic subgroup of $M_\s$ associated to the subset of simple roots
$$\Pi(\rho)_\s := \left\{ \alpha \in \Pi_{\s}:  \begin{array}{l}\diamond~ \langle \Pi_{M}, \alpha^{\vee}\rangle = 0 \\ \diamond~ \alpha^\vee(x) \textnormal{ acts trivially on } \rho \textnormal{ for all }x\in F^\times\end{array} \right\},$$
and ${}^e\rho$ is an extension of $\rho$ to the Levi factor $M(\rho)_\s$ of $P(\rho)_\s$. Therefore 
$$\sigma\otimes\chi_\psi \cong \Ind_{P(\rho)_\s^{-}}^{M_\s}\left({}^e\rho \otimes \St_{Q^{-}}^{P(\rho)_\s^{-}}\right)\otimes\chi_{\psi} \cong \Ind_{\tP(\rho)_\s^{-}}^{\tM_\s}\left({}^e \rho\otimes\chi_{\psi} \otimes \St_{Q^{-}}^{P(\rho)_\s^{-}}\right).$$  
Inducing from $\tP_{\s}^{-}$ to $\tG$ and using transitivity of induction, we obtain
$$\Ind_{\tP_\s^{-}}^{\tG}(\sigma\otimes\chi_{\psi}) \cong \Ind_{\tP(\rho)_\s^{-}\ltimes N_\s^{-,*}}^{\tG}\left({}^e\rho\otimes\chi_{\psi} \otimes \St_{Q^{-}\ltimes N_\s^-}^{P(\rho)_\s^{-}\ltimes N_\s^-}\right).$$

Let $\widetilde{\rho} := \rho\otimes\chi_{\psi}$; then $\widetilde{\rho}$ is a supersingular representation of $\tM$ in the sense of Definition \ref{supersingular}, and ${}^e\rho\otimes\chi_{\psi} \cong {}^e(\rho\otimes\chi_{\psi}) \cong {}^e\widetilde{\rho}$ where the latter is an extension of $\widetilde{\rho}$ to $\tM(\rho)_\s$. Since the action of $\tT \cap M_{\alpha_n}'$ on $\widetilde{\rho}$ is nontrivial, $\Pi(\rho)_\s$ is equal to $\Pi(\widetilde{\rho})$ and so 
$$\Ind_{\tP_\s^{-}}^{\tG}(\sigma\otimes\chi_{\psi}) \cong \Ind_{\tP(\widetilde{\rho})^-}^{\tG}\left({}^e\widetilde{\rho} \otimes \St_{Q^{-}\ltimes N_\s^-}^{P(\widetilde{\rho})^{-}}\right) = I(\tP^{-}\ltimes N_\s^{-,*}, \widetilde{\rho},\tQ^{-}\ltimes N_\s^{-,*}).$$
\end{proof}

We also have the following description of genuine principal series representations of $\tG$:

\begin{cor}
\label{irredps}
Let $\sigma$ be a genuine character of $\tT$. 
\begin{enumerate}
\item The length of $\Ind_{\tB^{-}}^{\tG}(\sigma)$ is at most $2^{n-1}$, and is equal to $2^{n-1}$ if and only if $\sigma \circ \widetilde{\alpha}^{\vee}(x) = 1$ for every $x\in F^\times$ and every short root $\alpha\in \Pi$.\item The representation $\Ind_{\tB^{-}}^{\tG}(\sigma)$ is irreducible if and only if, for every short root $\alpha\in \Pi$, there exists $x_\alpha\in F^\times$ such that $\sigma \circ \widetilde{\alpha}^{\vee}(x_\alpha) \neq 1$. 
\item Let $\sigma'$ be a genuine character of $\tT$ such that $\sigma \neq \sigma'$. Then $\Ind_{\tB^{-}}^{\tG}(\sigma)$ and $\Ind_{\tB^{-}}^{\tG}(\sigma')$ are inequivalent. 
\end{enumerate}
\end{cor}

\begin{proof} 
Corollary \ref{parabindjh} implies that the length of $\Ind_{\tB^{-}}^{\tG}(\sigma)$ is equal to the number of subsets of
$$\{\alpha \in \Pi_{\s}: \, \text{the image of $\widetilde{\alpha}^{\vee}$ acts trivially on $\sigma$}\},$$
giving (1) and (2) (alternatively, one may use Lemma \ref{msirred} and Theorem 8.7 of \cite{herzig:modpgln}). Part (3) follows from a comparison of the composition factors using Corollary \ref{parabindjh}, together with injectivity of the map in Theorem \ref{classn}. 
\end{proof}

\begin{remark}
Recall the genuine character $\chi_{\psi}$ of $\tM_{\s}$ from $\S$\ref{twistrep}, and denote by the same symbol the restriction of that character to $\tT$. The characters of $T$ and the genuine characters of $\tT$ are in bijection via $\xi \longmapsto \xi \otimes \chi_{\psi}$. For a fixed nontrivial additive $\bbC$-character $\psi$ of $F$, one can then express the criteria of Corollary \ref{irredps} in terms of $\xi$. In fact, since $\det_{M_\s}(\alpha_i^{\vee}(x)) = 1$ for each $1 \leq i \leq n-1$ and $x \in F^\times$, the value of $\chi_{\psi}$ on the image of $\widetilde{\alpha}_i^{\vee}$ is independent of $\psi$ for each $1 \leq i \leq n-1$. Therefore, the length of $\Ind_{\tB^{-}}^{\tG}(\xi \otimes \chi_{\psi})$ is independent of $\psi$. On the other hand, Corollary \ref{irredps}(3) shows that $\Ind_{\tB^{-}}^{\tG}(\xi \otimes \chi_{\psi})$ and $\Ind_{\tB^{-}}^{\tG}(\xi \otimes \chi_{\psi_a})$ are equivalent if and only if $\chi_{\psi} \circ \widetilde{\alpha}_n^{\vee}(x) = \chi_{\psi_{a}} \circ \widetilde{\alpha}_n^{\vee}(x)$ for every $x\in F^\times$, i.e., if and only if $a \in (F^\times)^2$. 
\end{remark}

\bibliographystyle{amsplain}
\bibliography{mp2n_mod_p}

\end{document}